\documentclass[a4paper,twoside]{article}
\usepackage{a4}
\usepackage{amssymb}
\usepackage{amsmath}
\usepackage{upref}
\usepackage{url}
\usepackage[active]{srcltx}
\usepackage[pagebackref,colorlinks,citecolor=blue,linkcolor=blue,urlcolor=blue]{hyperref}
\allowdisplaybreaks[2] 
\newcount\minutes \newcount\hours
\hours=\time
\divide\hours 60
\minutes=\hours
\multiply\minutes -60
\advance\minutes \time
\newcommand{\klockan}{\the\hours:{\ifnum\minutes<10 0\fi}\the\minutes}
\newcommand{\tid}{\today\ \klockan}
\newcommand{\prtid}{\smash{\raise 10mm \hbox{\LaTeX ed \tid}}}
\renewcommand{\prtid}{}
\makeatletter
\def\sectionmark#1{} 
\def\subsectionmark#1{}
\newcommand{\sectnr}{\ifnum \c@secnumdepth >\z@
                 \thesection.\hskip 1em\relax \fi}
\def\@evenhead{\footnotesize\rm\thepage\hfil\leftmark\hfil\llap{\prtid}}
\def\@oddhead{\footnotesize\rm\rlap{\prtid}\hfil\rightmark\hfil\thepage}
\def\tableofcontents{\section*{Contents} 
 \@starttoc{toc}}
\makeatother
\makeatletter
\def\@biblabel#1{#1.}
\makeatother
\makeatletter
\let\Thebibliography=\thebibliography
\renewcommand{\thebibliography}[1]{\def\@mkboth##1##2{}\Thebibliography{#1}
\addcontentsline{toc}{section}{References}
\frenchspacing 
\setlength{\@topsep}{0pt}
\setlength{\itemsep}{0pt}%
\setlength{\parskip}{0pt plus 2pt}%
}
\makeatother
\makeatletter
\def\mdots@{\mathinner.\nonscript\!.%
 \ifx\next,.\else\ifx\next;.\else\ifx\next..\else
 \nonscript\!\mathinner.\fi\fi\fi}
\let\ldots\mdots@
\let\cdots\mdots@
\let\dotso\mdots@
\let\dotsb\mdots@
\let\dotsm\mdots@
\let\dotsc\mdots@
\def\vdots{\vbox{\baselineskip2.8\p@ \lineskiplimit\z@
    \kern6\p@\hbox{.}\hbox{.}\hbox{.}\kern3\p@}}
\def\ddots{\mathinner{\mkern1mu\raise8.6\p@\vbox{\kern7\p@\hbox{.}}%
    \raise5.8\p@\hbox{.}\raise3\p@\hbox{.}\mkern1mu}}
\makeatother
\makeatletter
\let\Enumerate=\enumerate
\renewcommand{\enumerate}{\Enumerate%
\setlength{\@topsep}{0pt}
\setlength{\itemsep}{0pt}%
\setlength{\parskip}{0pt plus 1pt}%
\renewcommand{\theenumi}{\textup{(\alph{enumi})}}%
\renewcommand{\labelenumi}{\theenumi}%
}
\let\endEnumerate=\endenumerate
\renewcommand{\endenumerate}{\endEnumerate\unskip}
\makeatother
\makeatletter
\def\@seccntformat#1{\csname the#1\endcsname.\quad}
\makeatother
\newcommand{\authortitle}[2]{\author{#1}\title{#2}\markboth{#1}{#2}}
\newcommand{\auth}[2]{{#2. #1}}
\newcommand{\art}[6]{{\sc #1, \rm #2, \it #3\/ \bf #4 \rm (#5), \mbox{#6}.}}
\newcommand{\artnopt}[6]{{\sc #1, \rm #2, \it #3\/ \bf #4 \rm (#5), \mbox{#6}}}
\newcommand{\artprep}[3]{{\sc #1, \rm #2, \it #3.}}
\newcommand{\artin}[3]{{\sc #1, \rm #2,  in #3.}}

\newcommand{\book}[3]{{\sc #1, \it #2, \rm #3.}}
\newcommand{\AND}{{\rm and }}
\RequirePackage{amsthm}
\newtheoremstyle{descriptive}%
  {\topsep}   
  {\topsep}   
  {\rmfamily} 
  {}          
  {\bfseries} 
  {.}         
  { }         
  {}          
\newtheoremstyle{propositional}%
  {\topsep}   
  {\topsep}   
  {\itshape}  
  {}          
  {\bfseries} 
  {.}         
  { }         
  {}          
\newtheoremstyle{remarkstyle}%
  {\topsep}   
  {\topsep}   
  {\rmfamily}  
  {}          
  {\itshape} 
  {.}         
  { }         
  {}          
\theoremstyle{propositional}
\newtheorem{thm}{Theorem}[section]
\newtheorem{prop}[thm]{Proposition}
\newtheorem{lem}[thm]{Lemma}
\newtheorem{cor}[thm]{Corollary}
\theoremstyle{descriptive}
\newtheorem{deff}[thm]{Definition}
\newtheorem{example}[thm]{Example}
\newtheorem{remark}[thm]{Remark}
\makeatletter
\renewenvironment{proof}[1][\proofname]{\par
  \pushQED{\qed}%
  \normalfont
  \trivlist
  \item[\hskip\labelsep
        \itshape
    #1\@addpunct{.}]\ignorespaces
}{%
  \popQED\endtrivlist\@endpefalse
}
\makeatother
\newcommand{\setm}{\setminus}
\renewcommand{\emptyset}{\varnothing}
\def\vint{\mathop{\mathchoice%
          {\setbox0\hbox{$\displaystyle\intop$}\kern 0.22\wd0%
           \vcenter{\hrule width 0.6\wd0}\kern -0.82\wd0}%
          {\setbox0\hbox{$\textstyle\intop$}\kern 0.2\wd0%
           \vcenter{\hrule width 0.6\wd0}\kern -0.8\wd0}%
          {\setbox0\hbox{$\scriptstyle\intop$}\kern 0.2\wd0%
           \vcenter{\hrule width 0.6\wd0}\kern -0.8\wd0}%
          {\setbox0\hbox{$\scriptscriptstyle\intop$}\kern 0.2\wd0%
           \vcenter{\hrule width 0.6\wd0}\kern -0.8\wd0}}%
          \mathopen{}\int}
{\catcode`p =12 \catcode`t =12 \gdef\eeaa#1pt{#1}}      
\def\accentadjtext#1{\setbox0\hbox{$#1$}\kern   
                \expandafter\eeaa\the\fontdimen1\textfont1 \ht0 }
\def\accentadjscript#1{\setbox0\hbox{$#1$}\kern 
                \expandafter\eeaa\the\fontdimen1\scriptfont1 \ht0 }
\def\accentadjscriptscript#1{\setbox0\hbox{$#1$}\kern   
                \expandafter\eeaa\the\fontdimen1\scriptscriptfont1 \ht0 }
\def\accentadjtextback#1{\setbox0\hbox{$#1$}\kern       
                -\expandafter\eeaa\the\fontdimen1\textfont1 \ht0 }
\def\accentadjscriptback#1{\setbox0\hbox{$#1$}\kern     
                -\expandafter\eeaa\the\fontdimen1\scriptfont1 \ht0 }
\def\accentadjscriptscriptback#1{\setbox0\hbox{$#1$}\kern 
                -\expandafter\eeaa\the\fontdimen1\scriptscriptfont1 \ht0 }
\def\itoverline#1{{\mathsurround0pt\mathchoice
        {\rlap{$\accentadjtext{\displaystyle #1}
                \accentadjtext{\vrule height1.593pt}
                \overline{\phantom{\displaystyle #1}
                \accentadjtextback{\displaystyle #1}}$}{#1}}
        {\rlap{$\accentadjtext{\textstyle #1}
                \accentadjtext{\vrule height1.593pt}
                \overline{\phantom{\textstyle #1}
                \accentadjtextback{\textstyle #1}}$}{#1}}
        {\rlap{$\accentadjscript{\scriptstyle #1}
                \accentadjscript{\vrule height1.593pt}
                \overline{\phantom{\scriptstyle #1}
                \accentadjscriptback{\scriptstyle #1}}$}{#1}}
        {\rlap{$\accentadjscriptscript{\scriptscriptstyle #1}
                \accentadjscriptscript{\vrule height1.593pt}
                \overline{\phantom{\scriptscriptstyle #1}
                \accentadjscriptscriptback{\scriptscriptstyle #1}}$}{#1}}}}
\def\itunderline#1{{\mathsurround0pt\mathchoice
        {\rlap{$\underline{\phantom{\displaystyle #1}
                \accentadjtextback{\displaystyle #1}}$}{#1}}
        {\rlap{$\underline{\phantom{\textstyle #1}
                \accentadjtextback{\textstyle #1}}$}{#1}}
        {\rlap{$\underline{\phantom{\scriptstyle #1}
                \accentadjscriptback{\scriptstyle #1}}$}{#1}}
        {\rlap{$\underline{\phantom{\scriptscriptstyle #1}
                \accentadjscriptscriptback{\scriptscriptstyle #1}}$}{#1}}}}
\newcommand{\Cp}{{C_p}}
\newcommand{\Cpmu}{{C_{p,\mu}}}
\DeclareMathOperator{\diam}{diam}
\DeclareMathOperator{\Div}{div}
\DeclareMathOperator{\capp}{cap}
\newcommand{\cp}{\capp_p}
\newcommand{\cpmu}{\capp_{p,\mu}}
\newcommand{\ctp}{\capp_{t_p}}
\newcommand{\grad}{\nabla}
\DeclareMathOperator{\Lip}{Lip}
\newcommand{\Lipc}{{\Lip_c}}
\DeclareMathOperator*{\essliminf}{ess\,lim\,inf}
\DeclareMathOperator*{\essinf}{ess\,inf}
\newcommand{\bdry}{\partial}
\newcommand{\bdy}{\bdry}
\newcommand{\loc}{_{\rm loc}}
\newcommand{\simge}{\gtrsim}
\newcommand{\simle}{\lesssim}
\newcommand{\al}{\alpha}
\newcommand{\alp}{\alpha}
\newcommand{\be}{\beta}
\newcommand{\ga}{\gamma}
\newcommand{\de}{\delta}
\newcommand{\eps}{\varepsilon}
\newcommand{\la}{\lambda}
\newcommand{\om}{\omega}
\newcommand{\Om}{\Omega}
\renewcommand{\phi}{\varphi}
\newcommand{\p}{{$p\mspace{1mu}$}}
\newcommand{\R}{\mathbf{R}}
\newcommand{\Sphere}{\mathbf{S}}
\def\cprime{{\mathsurround0pt$'$}}
\newcommand{\limplus}{{\mathchoice{\vcenter{\hbox{$\scriptstyle +$}}}
  {\vcenter{\hbox{$\scriptstyle +$}}}
  {\vcenter{\hbox{$\scriptscriptstyle +$}}}
  {\vcenter{\hbox{$\scriptscriptstyle +$}}}
}}
\newcommand{\Np}{N^{1,p}}
\newcommand{\Npoo}{N^{1,p_0}_0}
\newcommand{\Nploc}{N^{1,p}\loc}
\newcommand{\ut}{\tilde{u}}
\newcommand{\lQo}{\itunderline{Q}_0}
\newcommand{\uQo}{\itoverline{Q}_0}
\newcommand{\lSo}{\itunderline{S}_0}
\newcommand{\uSo}{\itoverline{S}_0}
\newcommand{\lQ}{\itunderline{Q}}
\newcommand{\uq}{\overline{q}}
\newcommand{\lqq}{\underline{q}} 
\newcommand{\us}{\itoverline{s}}
\newcommand{\ls}{\itunderline{s}}
\newcommand{\uth}{\itoverline{\theta}}
\newcommand{\uthtilde}{\tilde{\theta}}
\newcommand{\lth}{\itunderline{\theta}}
\newcommand{\uTh}{\overline{\Theta}}
\newcommand{\lTh}{\underline{\Theta}}
\newcommand{\uqo}{\uq_0}
\newcommand{\lqo}{\lqq_0}
\newcommand{\uso}{\us_0}
\newcommand{\lso}{\ls_0}
\newcommand{\utho}{\uth_0}
\newcommand{\ltho}{\lth_0}
\newcommand{\Ga}{\Gamma}
\newcommand{\Lploc}{L^p\loc}
\newcommand{\qhat}{\hat{q}}
\numberwithin{equation}{section}
\newcommand{\eqv}{\ensuremath{\mathchoice{\quad \Longleftrightarrow \quad}{\Leftrightarrow}}
                {\Leftrightarrow}{\Leftrightarrow}}
\newenvironment{ack}{\medskip{\it Acknowledgement.}}{}

\begin{document}

\authortitle{Anders Bj\"orn, Jana Bj\"orn
    and Juha Lehrb\"ack}
{Volume growth, \p-parabolicity and 
integrability of \p-harmonic Green functions}
\title{Volume growth, capacity estimates, \p-parabolicity  and 
sharp integrability properties \\ of 
\p-harmonic Green functions}
\author{
Anders Bj\"orn \\
\it\small Department of Mathematics, Link\"oping University, SE-581 83 Link\"oping, Sweden\\
\it \small anders.bjorn@liu.se, ORCID\/\textup{:} 0000-0002-9677-8321
\\
\\
Jana Bj\"orn \\
\it\small Department of Mathematics, Link\"oping University, SE-581 83 Link\"oping, Sweden\\
\it \small jana.bjorn@liu.se, ORCID\/\textup{:} 0000-0002-1238-6751
\\
\\
Juha Lehrb\"ack \\
\it\small Department of Mathematics and Statistics, University of Jyv\"askyl\"a,\\
\it\small P.O. Box 35 (MaD), FI-40014 University of Jyv\"askyl\"a, Finland\/{\rm ;}
\it \small juha.lehrback@jyu.fi \\
\it \small ORCID\/\textup{:} 0000-0002-2880-7979
\\
}

\date{Preliminary version, \today}
\date{}

\maketitle

\noindent{\small
{\bf Abstract}. 
In a complete metric space equipped with a doubling measure
supporting a \p-Poincar\'e inequality,
we prove sharp growth and integrability results
for \p-harmonic Green functions 
and their minimal \p-weak upper gradients.
We show that these properties are determined
by the growth of the underlying measure near the singularity.
Corresponding results are obtained also for more general \p-harmonic
functions with poles, as well as for singular solutions of elliptic differential
equations in divergence form on weighted $\R^n$ and on manifolds.

The proofs are based on 
a new general capacity estimate for annuli,
which implies precise
pointwise estimates for 
\p-harmonic Green functions. 
The capacity estimate is valid under considerably milder assumptions
than above. We also use it, under these milder assumptions, to
characterize singletons of zero capacity and
the \p-parabolicity of the space. 
This generalizes and improves
earlier results that have been important especially in
the context of Riemannian manifolds.
}

\medskip

\noindent {\small \emph{Key words and phrases}:
capacity,
doubling measure,
Green function,
integrability,
metric space,
\p-harmonic function,
\p-hyperbolic space,
Poincar\'e inequality,
\p-parabolic space,
singular function,
volume growth,
weak upper gradient.
}

\medskip


\section{Introduction}

In this paper we study the growth and $L^\tau$-integrability
of \p-harmonic Green (and singular) functions in metric measure spaces, as well as
$L^t$-integrability of their minimal \p-weak upper gradients, 
with $1<p<\infty$.
We show that these properties are determined
by the growth of the measure  near the singularity.
We also obtain corresponding results for more general \p-harmonic
functions with poles, as well as for singular solutions of elliptic differential
equations in divergence form on weighted $\R^n$ and on manifolds.

Recall that
$u$ is a \emph{\p-harmonic Green function} in 
a bounded domain $\Om\subset \R^n$
with singularity at $x_0\in\Om$ if
\begin{equation} \label{eq-Green-Rn}
\Delta_p u:= \Div(|\nabla u|^{p-2}\nabla u) = -\delta_{x_0}
\quad \text{in } \Om
\end{equation}
with zero boundary values on $\bdy\Om$ (in Sobolev sense),
where $\delta_{x_0}$ is the Dirac measure 
at $x_0$.
Such a function $u$ is 
\p-harmonic (i.e.\ $\Delta_p u=0$)
in $\Omega\setminus\{x_0\}$
and \p-superharmonic (i.e.\ $\Delta_p u \le 0$) in the whole domain $\Omega$.
If $1<p\le n$, then also $\lim_{x \to x_0}u(x)=\infty$.

In a metric measure space $X=(X,d,\mu)$
there is (a priori) no equation available for defining \p-harmonic
functions, and they are instead defined as local minimizers of the 
\p-energy integral
\[
     \int g_u^p \, d\mu,         
\]
where $g_u$ is the minimal \p-weak upper gradient of $u$.
For example, on $\R^n$ we have $g_u=|\grad u|$ and 
these definitions 
of \p-harmonic functions and \p-harmonic Green functions are equivalent to the definitions
using the \p-Laplace operator $\Delta_p u$.

Let $\Omega\subset X$ be a bounded domain 
and assume that $x_0\in\Om$ with \p-capacity $\Cp(\{x_0\})=0$. 
Following our earlier paper \cite{BBLehGreen}, we
say that $u$ is a \emph{singular function} in $\Om$
with singularity at $x_0$ if 
$u$ is \p-harmonic in $\Omega\setminus\{x_0\}$ and
\p-superharmonic in $\Omega$, 
$u=0$ on $\bdy\Om$ in the Sobolev sense
and $\lim_{x \to x_0}u(x)=\infty$. 
A \emph{Green function} is then a precisely
scaled singular function.
See Definition~\ref{deff-sing} for
exact definitions.
Earlier definitions
are due to Holo\-pai\-nen~\cite{Ho} for manifolds,
Heinonen--Kilpel\"ainen--Martio~\cite[Section~7.38]{HeKiMa} for weighted $\R^n$ and 
Holo\-pai\-nen--Shan\-mu\-ga\-lin\-gam~\cite{HoSha} for metric spaces.

\medskip

\emph{
Throughout the paper, we fix $1<p<\infty$ and a point $x_0 \in X$
and write $B_r=B(x_0,r)$.
For the rest of the introduction, we also assume that 
$X$ is a complete metric space equipped with
a doubling measure $\mu$ that supports a \p-Poincar\'e inequality.}

\medskip

The following result summarizes some 
of the main results in this paper,
many of which are new also in weighted $\R^n$ and on manifolds.
Let
\begin{align*}
  \uso
         & =\inf \{s>0 : \text{there is $C_s>0$ so that } 
        \mu(B(x_0,r))  \ge C_s r^s 
        \text{ for } 0 < r  \le 1
        \}, \nonumber \\
\tau_p&= \begin{cases}
     \displaystyle \frac{\uso(p-1)}{\uso-p}, & \text{if } p < \uso, \\
     \infty, & \text{if } p = \uso,
     \end{cases}
\qquad \text{and} \qquad
t_p=\frac{\uso(p-1)}{\uso-1}.
\end{align*}

\begin{thm} \label{thm-intro}
Let $\Om \subset X$ be a bounded domain containing $x_0$, and 
$u$ be a singular or Green function in $\Om$ with singularity at $x_0$.
Assume that $\Cp(\{x_0\})=0$.
Then the following are true\/\textup{:}
\begin{enumerate}
\item \label{intro-a}
$p \le \uso$ and $u$ is unbounded\textup{;}
\item \label{intro-b}
  $u \in L^\tau(\Om)$ for all $0 < \tau < \tau_p$\textup{;}
\item \label{intro-c}
  $u \notin L^\tau(\Om)$ if $\tau > \tau_p$\textup{;}
\item \label{intro-d}
$g_u\in L^t(\Om)$ for all 
$0<t<t_p$\textup{;}
\item \label{intro-e}
if $p = \uso$, then 
$g_u \in L^t(\Om)$ if and only if $0<t<p$\textup{;}
\item \label{intro-f}
if $p <\uso$, then $g_u\notin L^t(\Om)$ 
whenever $\mu$ supports a $t$-Poincar\'e inequality\/
\textup{(}at $x_0$ and for small radii\/\textup{)},
$t>t_p$ and $t \ge 1$.
\end{enumerate}
\end{thm}

The case $\Cp(\{x_0\})>0$
is not of interest here, since in this case
every singular (and Green) function $u$
in $\Om$ with singularity at $x_0$ is bounded and $g_u \in L^p(\Om)$;
see Theorem~\ref{thm-Green-cp-x0},
which also 
shows that the Green function 
(with singularity at $x_0$) is unbounded if and only if
\begin{equation}  \label{eq-div-int}
\int_0^1 \biggl( \frac{\rho}{\mu(B_\rho)} \biggr)^{1/(p-1)} \,d\rho = 
\infty.
\end{equation}

In the borderline case $\tau=\tau_p$ we completely characterize
when $u \in L^{\tau_p}(\Om)$
in terms of integrals similar to the one in 
\eqref{eq-div-int},
see Theorem~\ref{thm-u-Q}.
We also provide sharp
results on the integrability of the minimal 
\p-weak upper gradient $g_u$ in the borderline case $t=t_p$,
see Theorem~\ref{thm-nonintegrability-gu}.
In the locally pointwise Ahlfors $Q$-regular case we obtain the following complete
characterization.  
In particular, it applies to Riemannian
manifolds with nonnegative Ricci curvature and to Carnot groups.

\begin{thm} \label{thm-intro-Q-reg}
Let $\Om \subset X$ be a bounded domain containing $x_0$, and 
$u$ be a singular or Green function in $\Om$ with singularity
at $x_0$.
Assume that $Q\ge p$ and that $\mu$ is Ahlfors $Q$-regular around $x_0$ 
for small radii, i.e. 
\[
               \mu(B(x_0,r)) \simeq r^Q,
               \quad \text{if } 0<r\le 1.
\]
Then in a neighbourhood of $x_0$,
\begin{equation} \label{eq-formula-Qreg}
u(x) \simeq \begin{cases}
      d(x,x_0)^{(p-Q)/(p-1)}, & \text{if } p < Q, \\
      -\log d(x,x_0),  &\text{if } p =Q,
\end{cases}
\end{equation}
and the following are true\/\textup{:}
\begin{enumerate}
\item \label{i-a}
$u \in L^\tau(\Om)$ if and only if 
$0<\tau< \tau_p:=Q(p-1)/(Q-p)$ {\rm(}where $\tau_Q=\infty${\rm)}\textup{;}
\item \label{i-b}
$g_u\in L^t(\Om)$ for all $0<t<t_p:= Q(p-1)/(Q-1)$\textup{;}
\item \label{i-c}
$g_u \notin L^{t}(\Om)$ if $t \ge \max\{1,t_p\}$ and $X$ supports a 
$t$-Poincar\'e inequality\/
\textup{(}at $x_0$ and for small radii\/\textup{)}.
\end{enumerate}
\end{thm}

Using the flexibility of our definition
of singular functions, with no a priori superlevel set requirements,
Theorem~\ref{thm-intro} implies the following very similar growth
properties for general \p-harmonic functions with poles.

\begin{thm} \label{thm-intro-gen-pharm-new}
Let $\Om \subset X$ be an open set containing $x_0$.
Assume that $u\ge0$ is a \p-harmonic function in $\Om \setm \{x_0\}$
such that $\lim_{x \to x_0} u(x)=\infty$.
Then $\Cp(\{x_0\})=0$
and the statements \ref{intro-a}--\ref{intro-f}
in Theorem~\ref{thm-intro} about {\rm(}non\/{\rm)}integrability 
hold true with  $L^\tau\loc(\Om)$ and  $L^t\loc(\Om)$ instead of 
$L^\tau(\Om)$ and  $L^t(\Om)$.

Moreover, there exists $R>0$ such that in a neighbourhood of $x_0$,
\begin{equation} \label{eq-Wolff-intro}
u(x) \simeq \inf_{B_{R}} u 
    + \int_{d(x,x_0)}^{R} \biggl( \frac{\rho}{\mu(B_\rho)} \biggr)^{1/(p-1)} \,d\rho.
\end{equation}
For Green and singular functions, $\inf_{B_{R}} u$ can be replaced by $0$.
\end{thm}

See also Theorem~\ref{thm-sum-A-superh}
for corresponding integrability properties of 
singular functions for elliptic differential
equations in divergence form on weighted $\R^n$.
The comparison \eqref{eq-Wolff-intro} can be seen as a Wolff potential estimate in terms of the 
Dirac measure $\de_{x_0}$, cf.\ Remark~\ref{rmk-Wolff}.
This is natural in view of 
\eqref{eq-Green-Rn} even though there need not be such an equation 
in the metric setting.

Existence of singular and Green functions
was proved in 
\cite[Theorem~1.3]{BBLehGreen} for bounded open sets $\Om$ 
with
$\Cp(X \setm \Om)>0$.
It was also shown therein that any two
Green functions in $\Om$ with the same singularity 
are comparable to each other and thus have the same growth behaviour
near the singularity; see Theorem~\ref{thm-main-intro}.
More explicitly,
for a Green function $u$ in $\Om$
with singularity at $x_0$, we have
by Theorem~\ref{thm-comp-u-R1-R2}
that
\begin{equation} \label{eq-Green-cp}
     u(x) \simeq \cp(B_r,\Om)^{1/(1-p)},
     \quad \text{if }
     0<d(x,x_0)=r<R,
\end{equation}
where $R$ depends only on $\Om$.
In many cases, estimate \eqref{eq-Green-cp}
can be expressed in terms of $r$ and $\mu(B_r)$
by using the critical exponents and exponent sets for the volume growth, 
studied in~\cite{BBLeh1}; see Section~\ref{sect-exponents} and Corollary~\ref{cor-BBL-conseq}.

The pointwise estimates and integrability properties 
of Green functions and their minimal \p-weak upper gradients in
Theorems~\ref{thm-intro}--\ref{thm-intro-gen-pharm-new} 
are based on \eqref{eq-Green-cp} 
and
the following new general capacity estimate, 
\begin{equation} \label{eq-cp-intro}
\cp(B_r,B_R)
\simeq \biggl( \int_r^R \biggl( \frac{\rho}{\mu(B_\rho)} \biggr)^{1/(p-1)}
           \,d\rho \biggr)^{1-p}
\end{equation}
for $0<2r\le R\le \tfrac14 \diam X$.
We prove \eqref{eq-cp-intro} in Theorem~\ref{thm-metaestimate-int-only}
only assuming that the Poincar\'e inequality and the 
doubling (and reverse-doubling) condition for $\mu$ hold for balls centred at 
$x_0$.
In Propositions~\ref{prop-meta-zero-cap} and~\ref{thm-p-parab}
we characterize when singletons have zero capacity
and when $X$ is \p-parabolic, 
by letting $r\to 0$ and $R\to\infty$ in~\eqref{eq-cp-intro}, respectively.

In the weighted linear case on $\R^n$ (with $p=2$),
Fabes--Jerison--Kenig~\cite[Lemma~3.1 and Theorem~3.3]{FaJeKe} 
obtained 
\eqref{eq-Green-cp} and \eqref{eq-cp-intro} already in 1982.
Nevertheless, as far as we know, even in this case 
the sharp integrability results as in Theorem~\ref{thm-intro}
do not appear in the existing literature.
Neither do the nonintegrability results, as in \ref{intro-c} and \ref{intro-f},
although they are well known for unweighted $\R^n$
(\ref{intro-f} follows from 
\ref{intro-c} and the Sobolev inequality,
see also (0.8) in Kichenassamy--V\'eron~\cite{KichVeron}).

Our general assumptions on doubling and \p-Poincar\'e inequality
are fulfilled on weighted $\R^n$
equipped with a \p-admissible weight as in 
Heinonen--Kilpel\"ainen--Martio~\cite{HeKiMa}, 
on Riemannian manifolds with nonnegative Ricci curvature,
on Carnot groups, 
and for vector fields satisfying the H\"ormander condition,
as well as in many other situations.
Thus, our results hold for \p-harmonic functions and 
corresponding subelliptic equations in all these settings, 
see Haj\l asz--Koskela~\cite[Sections~10--13]{HaKo} 
for further details.
Moreover, as in \cite[Section~11]{BBLehGreen}
the assumptions can be relaxed
to similar local assumptions.

The exponents in Theorem~\ref{thm-intro} are often better
than in the integrability results for general \p-superharmonic functions
from Heinonen--Kilpel\"ainen--Martio~\cite[Theorem~7.46]{HeKiMa}
(on weighted $\R^n$) and 
Kin\-nu\-nen--Martio~\cite[Section~5]{KiMa} (on metric spaces). 
This happens e.g.\ if the local dimension $\uso$ at $x_0$ is smaller than
the global dimension of the space, provided by the doubling property of $\mu$.
For example, the $1$-admissible weight 
$w(x)=|x|^{-\alp}$ on $\R^n$, with $0<\alp<n$,
has $\uso=n-\alp$ at $x_0$, while the general
integrability for \p-superharmonic functions is dictated by the dimension $n$
(see also Example~\ref{ex-power}).
In the globally Ahlfors $Q$-regular case, i.e.\ when
\[
  \mu(B(x,r)) \simeq r^Q
  \quad \text{for all } x \in X \text{ and } r>0,
\]
the integrability conditions in Theorem~\ref{thm-intro-Q-reg}
follow from the general integrability
results in Kinnunen--Martio~\cite{KiMa},
but even in this case the nonintegrability conditions
in Theorem~\ref{thm-intro-Q-reg} 
seem to be new.
For corresponding singular solutions in 
Carnot--Carath\'eodory spaces, the nonintegrability in
Theorem~\ref{thm-intro}\,\ref{intro-c} follows from
Capogna--Danielli--Garofalo~\cite[Corollary~6.1]{CaDaGa2}.
See also~\cite[Theorem~7.1]{CaDaGa2} for pointwise
estimates related to~\eqref{eq-Green-cp}.

Danielli--Garofalo--Marola~\cite[Corollary~5.4]{DaGaMa}
obtained positive integrability results as in 
Theorem~\ref{thm-intro}\,\ref{intro-b}
and \ref{intro-d} for singular functions defined as 
in Holo\-pai\-nen--Shan\-mu\-ga\-lin\-gam~\cite{HoSha}, 
under some additional assumptions on $X$.
However they had smaller ranges of $p$, $\tau$ and $t$,
see Remark~\ref{rmk-DMG} and 
the comments after Theorem~\ref{thm-integrability-gu}.
On the other hand, as shown by \ref{intro-c}, \ref{intro-e} 
and \ref{intro-f}, the ranges in Theorem~\ref{thm-intro} 
are optimal up to the borderline cases.

Estimates \eqref{eq-Green-cp} and \eqref{eq-cp-intro} generalize 
and improve many results obtained earlier
mainly in the setting of Riemannian manifolds and under additional 
geometric and curvature assumptions. 
For example,  
estimates using chains of balls along geodesics were obtained 
in Holo\-pai\-nen--Koskela~\cite{HoKo}, while in
Coulhon--Holo\-pai\-nen--Saloff-Coste~\cite[Theorem~3.5]{CoHoSC}, 
capacity was estimated by means of the \p-isometric profile.  
See also Holo\-pai\-nen~\cite[p.~329]{HoDuke}.
Different but equivalent formulas for $A_p$-weighted $\R^n$ were 
given in Heinonen--Kilpel\"ainen--Martio~\cite[Theorems~2.18 and~2.19]{HeKiMa}.

Even though we only deal with Green and singular functions on 
bounded domains, 
the capacity estimate~\eqref{eq-cp-intro} has
consequences for the existence of global
Green functions as well.
More precisely, a complete noncompact (sub)Riemannian manifold 
is  called \p-\emph{parabolic} if it does \emph{not} carry a global Green
function.
The property of \p-parabolicity has applications for quasiconformal 
mappings and Picard theorems, 
and has been extensively studied in e.g.\ 
Coulhon--Holo\-pai\-nen--Saloff-Coste~\cite{CoHoSC}, 
Gri\-gor\cprime yan~\cite{Grig2} ($p=2$) and 
Holo\-pai\-nen~\cite{Ho} and~\cite{HoDuke}. 
In the manifold setting,
it is known that \p-parabolicity is implied
by the condition
\begin{equation}   \label{eq-int-parab-intro}
\int_{r_0}^\infty \biggl( \frac{\rho}{\mu(B_\rho)} \biggr)^{1/(p-1)} \,d\rho
= \infty,
\end{equation}
see e.g.\ \cite[Corollary~3.2]{CoHoSC},
\cite[Theorem~7.3]{Grig2} ($p=2$), \cite[Proposition~1.7]{HoDuke} and
Kesel\cprime man--Zorich~\cite{KeZor} ($p=n$).
The converse is in general not true 
(by \cite[p.~322]{HoDuke} or Varopoulos~\cite{Varo}),
but has been proved in some 
(sub)Riemannian manifolds.
In particular, \p-parabolicity and \eqref{eq-int-parab-intro}
are equivalent in Riemannian manifolds with nonnegative Ricci 
curvature or, more generally, satisfying a global doubling condition 
and a global Poincar\'e inequality,
see \cite[Proposition~3.4]{CoHoSC}, \cite[Corollary~4.12]{HoDuke}
and~\cite[Theorem~1.7]{HoKo} for more details.

One of the well-known equivalent characterizations of
\p-parabolicity of Riemannian manifolds is that all balls have
global variational \p-capacity 
zero, see~\cite[Theorem~5.1]{Grig2} ($p=2$), 
Holo\-pai\-nen~\cite[Theorem~3.27]{Ho} and \cite[p.~322]{HoDuke}.
This property is used as the definition of \p-parabolicity in
metric spaces by Holo\-pai\-nen--Koskela~\cite[p.~3428]{HoKo} and 
Holo\-pai\-nen--Shanmugalingam~\cite[Definition~3.13]{HoSha}.
Following the same definition, we show in Theorem~\ref{thm-p-parab} 
that under rather mild assumptions, an unbounded metric space 
is \p-parabolic if and only if \eqref{eq-int-parab-intro} holds.
This recovers and complements the sufficient condition proved in
\cite[Proposition~2.3]{HoKo} and generalizes several of the above
results.
Recall also that an unbounded space is said to be
\p-\emph{hyperbolic} if it is not \p-parabolic.

The outline of the paper is as follows.
In Section~\ref{sect-prelim} we recall basic definitions and assumptions
related to the analysis on metric spaces and in Section~\ref{sect-exponents}
we introduce the pointwise exponent sets, which govern the volume
growth near $x_0$. The capacity estimate~\eqref{eq-cp-intro}
is proved in Section~\ref{sect-meta}, where we also study some of
its consequences, while the applications of~\eqref{eq-cp-intro}
for the \p-parabolicity and capacity of singletons are
discussed in Section~\ref{sect-cap-zero}.

Background material on \p-(super)harmonic functions
as well as the definitions of singular and Green functions
are given in Section~\ref{sect-harm}.
Note that in the rest of the paper, we suppress the dependence on $p$ 
and write ``superharmonic'' instead of ``\p-superharmonic'', but keep the
term ``\p-harmonic''.
The pointwise behaviour of Green (and singular) functions near the
singularity is studied in Section~\ref{sect-ptwise-est}, where
we also show comparability of Green functions
(having the same singularity $x_0$) on comparable open sets
with comparable measures.

General integrability properties of superharmonic functions,
recalling and extending the results in Kinnunen--Martio~\cite{KiMa},
are reviewed in Section~\ref{sect-int-superh}.
Sections~\ref{sect-int-Green} and~\ref{sect-int-gu}
contain our main results concerning the (non)integrability
of Green functions and their minimal \p-weak upper gradients, respectively.
In particular, Theorems~\ref{thm-intro} and~\ref{thm-intro-Q-reg}
are proved at the end of Section~\ref{sect-int-gu}.
Some examples, based on radial weights on $\R^n$
and complementing the general integrability results,
are given in Section~\ref{sect-radial-weights}.
In Section~\ref{sect-poles}
we generalize the growth and integrability results to
\p-harmonic functions having a pole 
at $x_0$, and explain the connection between estimate~\eqref{eq-Wolff-intro}
and the Wolff potential.
Theorem~\ref{thm-intro-gen-pharm-new}
is proved at the end of Section~\ref{sect-poles}.
Finally, in Section~\ref{sect-elliptic-eq} we discuss 
how the (non)integrability results for Green functions can be
extended to singular functions for elliptic differential
equations in divergence form on weighted $\R^n$
and on Riemannian manifolds.

\begin{ack}
A.B. and J.B.  were supported by the Swedish Research Council,
grants 2016-03424 resp.\ 621-2014-3974 and 2018-04106.
Part of this research was done during
several visits of J.~L. to Link\"oping University; 
he is grateful for the support
and hospitality.
\end{ack}

\section{Preliminaries}
\label{sect-prelim}

We assume throughout the paper that $1 < p<\infty$ 
and that $X=(X,d,\mu)$ is a metric space equipped
with a metric $d$ and a positive complete  Borel  measure $\mu$ 
such that $0<\mu(B)<\infty$ for all balls $B \subset X$. 
Under these assumptions, $X$ is separable.
The $\sigma$-algebra on which $\mu$ is defined
is obtained by the completion of the Borel $\sigma$-algebra.
To avoid pathological situations we assume that $X$ contains
at least two points.
We also write $B(x,r)=\{y \in X : d(x,y)<r\}$.

Next we are going to introduce the
necessary background on Sobolev spaces and capacities in metric spaces.
Proofs of most of the results mentioned here
can be found in the monographs
Bj\"orn--Bj\"orn~\cite{BBbook} and
Heinonen--Koskela--Shanmugalingam--Tyson~\cite{HKSTbook}.

A \emph{curve} is a continuous mapping from an interval;
it is \emph{rectifiable} if it has finite length,
in which case it can be  parameterized by its arc length $ds$.
A property holds for \emph{\p-almost every curve}
if it fails only for a curve family $\Ga$ with zero \p-modulus, 
i.e.\ there is  $\rho\in L^p\loc(X)$ such that 
$\int_\ga \rho\,ds=\infty$ for every 
$\ga\in\Ga$.

A measurable function $g\colon X \to [0,\infty]$ is a 
\emph{\p-weak upper gradient}
of $u\colon X \to [-\infty,\infty]$ if for \p-almost every 
nonconstant compact rectifiable
curve
$\gamma\colon  [0,l_{\gamma}] \to X$,
\begin{equation}   \label{eq-def-u-grad}
         |u(\gamma(0)) - u(\gamma(l_{\gamma}))| \le \int_{\gamma} g\,ds,
\end{equation}
where we follow the convention that the left-hand side is $\infty$ 
whenever at least one of the 
terms therein is $\pm \infty$.
Weak upper gradients were
introduced by Koskela--MacManus~\cite{KoMc}, 
see also Heinonen--Koskela~\cite{HeKo98}.
If $u$ has a \p-weak upper gradient in $\Lploc(X)$, then
it has an a.e.\ unique \emph{minimal \p-weak upper gradient} $g_{u} \in \Lploc(X)$
in the sense that 
$g_{u} \le g$ a.e.\ for 
every \p-weak upper gradient $g \in \Lploc(X)$ of $u$.

Following Shanmugalingam~\cite{Sh-rev}, 
we define a version of Sobolev spaces on 
$X$.
For a measurable function $u\colon X\to [-\infty,\infty]$, let 
\[
        \|u\|_{\Np(X)} = \biggl( \int_X |u|^p \, d\mu 
                + \inf_g  \int_X g^p \, d\mu \biggr)^{1/p},
\]
where the infimum is taken over all \p-weak upper gradients of $u$.
The \emph{Newtonian space} on $X$ is 
\[
        \Np (X) = \{u: \|u\|_{\Np(X)} <\infty \}.
\]

The space $\Np(X)/{\sim}$, where $u \sim v$ if and only if $\|u-v\|_{\Np(X)}=0$,
is a Banach space and a lattice. 
In this paper we assume that functions in $\Np(X)$ are defined everywhere,
not just up to an equivalence class in the corresponding function space.
This is needed for \eqref{eq-def-u-grad} in
the definition of \p-weak upper gradients to make sense.
For an open set $\Om\subset X$, the Newtonian space $\Np(\Om)$ is defined by
considering $(\Om,d|_\Om,\mu|_\Om)$ as a metric space in its own right. 
Moreover,
$u \in \Nploc(\Om)$ if
for every $x \in \Om$ there exists $r_x>0$ such that 
$B(x,r_x)\subset\Om$ and $u \in \Np(B(x,r_x))$.
The space $\Lploc(\Om)$ is defined similarly.
If $u,v \in \Nploc(X)$, then $g_u=g_v$ a.e.\ in $\{x \in X : u(x)=v(x)\}$.
In particular $g_{\min\{u,c\}}=g_u \chi_{\{u < c\}}$ for $c \in \R$.

The \emph{Sobolev capacity} of an arbitrary set $E\subset X$ is
\[
\Cp(E) =\inf_u\|u\|_{\Np(X)}^p,
\]
where the infimum is taken over all $u \in \Np(X)$ such that
$u\geq 1$ on $E$.
The capacity is the correct gauge 
for distinguishing between two Newtonian functions. 
If $u \in \Nploc(X)$,  
then $u \sim v$ if and only if $u=v$ q.e.\ 
(quasieverywhere),
that is $\Cp(\{x: u(x) \ne v(x)\})=0$.
Moreover, 
if $u,v \in \Nploc(X)$ and $u= v$ a.e., then $u=v$ q.e.
Both the Sobolev and the following variational capacity 
are countably subadditive.

For an open set $\Om \subset X$, let
\[
\Np_0(\Om) 
  =\{u|_{\Om} : u \in \Np(X) \text{ and }
        u=0 \text{ on } X \setm \Om\}.
\]
The \emph{variational capacity} of $E \subset \Om$ with respect to $\Om$ is
\[ 
\cp(E,\Om) = \inf_u\int_{\Om} g_u^p\, d\mu,
\] 
where the infimum is taken over all $u \in \Np_0(\Om)$
such that $u \ge 1$ in $E$.
One can equivalently take the above infimum over all 
$u \in \Np(X)$ such that
$u=1$  on $E$ and $u=0$ on $X \setm \Om$;
we call such $u$ \emph{admissible} for the capacity
$\cp(E,\Om)$.
Similarly, whenever convenient, $u\in\Np_0(\Om)$ will be regarded 
as extended by $0$ outside $\Om$.

The measure
$\mu$ is \emph{\textup{(}globally\/\textup{)} doubling} if  there is a constant $C>0$ such that 
for all balls $B\subset X$
we have
\[ 
  \mu(2B)\le C \mu(B),
\] 
where $\la B(x,r)=B(x,\la r)$  for $\la>0$.
If $X$ is complete and $\mu$ is doubling, then $X$ is also \emph{proper},
i.e.\ sets which are closed and bounded are compact.

The space $X$ (or the measure $\mu$) supports a 
\emph{\textup{(}global\/\textup{)} \p-Poincar\'e inequality} if
there exist constants $C>0$ and $\lambda \ge 1$
such that for all balls $B=B(x,r)\subset X$, 
all integrable functions $u$ on $X$, and all 
\p-weak upper gradients $g$ of $u$, 
\begin{equation}  \label{eq-deff-PI}
        \vint_{B} |u-u_B| \,d\mu
        \le C r \biggl( \vint_{\lambda B} g^{p} \,d\mu \biggr)^{1/p},
\end{equation}
where $ u_B :=\vint_B u \,d\mu 
:= \int_B u\, d\mu/\mu(B)$.
If $X$ supports a Poincar\'e inequality, then $X$ is connected.

If $X=\R^n$ is equipped with $d\mu=w\,dx$, then 
$w\ge0$ is a \p-admissible weight
in the sense of 
Heinonen--Kilpel\"ainen--Martio~\cite{HeKiMa}
if and only if $\mu$ is a doubling measure which
supports a \p-Poincar\'e inequality,
see
Corollary~20.9 in~\cite{HeKiMa} (which is only in the second edition)
and Proposition~A.17 in~\cite{BBbook}.
In this case,
$\Np(\R^n)$ and $\Np(\Om)$ are the 
refined Sobolev spaces 
defined in~\cite[p.~96]{HeKiMa},
and moreover the above
Sobolev and variational capacities
coincide with those in \cite{HeKiMa};
see Bj\"orn--Bj\"orn~\cite[Theorem~6.7\,(ix) and Appendix~A.2]{BBbook} 
and~\cite[Theorem~5.1]{BBvarcap}.
The situation is similar
on Riemannian  manifolds with nonnegative Ricci curvature and 
on Carnot groups 
equipped with their natural measures; 
see Haj\l asz--Koskela~\cite[Sections~10 and~11]{HaKo}
for further details.

Throughout the paper, we write $Y \simle Z$ if there is an implicit
constant $C>0$ such that $Y \le CZ$.
We also write $Y \simge Z$ if $Z \simle Y$,
and $Y \simeq Z$ if $Y \simle Z \simle Y$.
Unless otherwise stated, we always allow the implicit comparison constants
to depend on the standard parameters, such as $p$, the doubling constant 
and the constants in the Poincar\'e inequality.

\section{Exponent sets}
\label{sect-exponents}

If $X$ is connected 
(which in particular holds
if it supports a Poincar\'e inequality) 
and $\mu$ is doubling, then there are positive constants 
$\lth\le \uth$ and $C$ such that 
\[ 
  \frac{1}{C} \Bigl(\frac{r}{R}\Bigr)^{\uth} \le 
  \frac{\mu(B(x,r))}{\mu(B(x,R))}
  \le C \Bigl(\frac{r}{R}\Bigr)^\lth
\] 
whenever $x\in X$ and $0<r\le R < 2 \diam X$.
It is easy to see that this condition is equivalent
to the corresponding noncentred conditions (3.1) and (3.2)
in \cite{BBbook}, provided that $\mu$ is doubling.
Example~\ref{ex-power} below 
shows that $\lth$ may need to be close to~$0$.
On the other hand, if $X$ is connected, then $\uth \ge 1$, see
Proposition~\ref{prop-uth-1}. 

The exponent $\uth$ plays a crucial role
in various results in the nonlinear potential theory,
such as in optimal exponents in Sobolev and
$(q,p)$-Poincar\'e inequalities, see Theorem~5.1 in 
Haj\l asz--Koskela~\cite{HaKo} or
\cite[Section~4.4]{BBbook}.
It also plays a prominent role in the integrability results
for superharmonic functions by 
Kinnunen--Martio~\cite{KiMa}, see Section~\ref{sect-int-superh}.

There may or may not be optimal values for $\uth$ and $\lth$
and it will therefore be useful to introduce 
the \emph{exponent sets}
\begin{align*}
  \lTh =\biggl\{\lth>0 :\, & \text{there is $C_{\lth}>0$ so that } 
  \frac{\mu(B(x,r))}{\mu(B(x,R))} \le C_{\lth} \Bigl(\frac{r}{R}\Bigr)^{\lth} \\
 &\text{for all } x\in X \text{ and } 0<r\le R < 2 \diam X
        \biggr\}, \\
  \uTh =\biggl\{\uth>0 :\, & \text{there is $C_{\uth}>0$ so that } 
  \frac{\mu(B(x,r))}{\mu(B(x,R))} \ge C_{\uth} \Bigl(\frac{r}{R}\Bigr)^{\uth} \\
 &\text{for all } x\in X \text{ and } 0<r\le R < 2 \diam X
        \biggr\}.
\end{align*}

Also the following pointwise \emph{exponent sets} at the fixed $x_0\in X$,
introduced in Bj\"orn--Bj\"orn--Lehrb\"ack~\cite{BBLeh1},
will be crucial in this paper.
Recall from the introduction that $B_r=B(x_0,r)$.
\begin{align*}
  \lQo  &=\biggl\{q>0 : \text{there is $C_q>0$ so that } 
        \frac{\mu(B_r)}{\mu(B_R)}  \le C_q \Bigl(\frac{r}{R}\Bigr)^q 
        \text{ for } 0 < r < R \le 1
        \biggr\}, \\
  \lSo
          &=\{s>0 : \text{there is $C_s>0$ so that } 
        \mu(B_r)  \le C_s r^s 
        \text{ for } 0 < r  \le 1
        \}, \\
  \uSo
          &=\{s>0 : \text{there is $C_s>0$ so that } 
        \mu(B_r)  \ge C_s r^s 
        \text{ for } 0 < r  \le 1
        \}, \\
  \uQo
       &=\biggl\{q>0 : \text{there is $C_q>0$ so that } 
       \frac{\mu(B_r)}{\mu(B_R)} 
       \ge C_q \Bigl(\frac{r}{R}\Bigr)^q 
       \text{ for } 0 < r < R \le 1
       \biggr\}.
\end{align*}

The subscript $0$ in the above definitions  stands for the fact
that the inequalities are required to hold for small radii.
All these sets are intervals and the reason for introducing them
as sets is that they may or may not contain their endpoints
\[
\ltho = \sup \lTh, 
\quad     \lqo = \sup \lQo, 
\quad     \lso = \sup \lSo, 
\quad   \uso = \inf \uSo,
\quad \uqo = \inf \uQo,
\quad 
\utho= \inf \uTh,
\]
respectively.
Nevertheless, it is always true that
\[
    \lTh \subset \lQo \subset \lSo,
\quad
    \uTh \subset \uQo \subset \uSo
\quad \text{and} \quad
    \ltho \le \lqo \le \lso \le \uso \le \uqo \le \utho.
\]

It was shown in \cite[Lemmas~2.4 and~2.5]{BBLeh1}
that the ranges $0 < r < R \le 1$  and  $0 < r  \le 1$,
in $\lQo$, $\lSo$, $\uSo$ and $\uQo$,
can equivalently be replaced by $0 < r < R \le R_0$ and  $0 < r  \le R_0$
for any fixed $R_0>0$ without changing the resulting
exponent sets.
The constants $C_q$ and $C_s$ may however change.
By Remark~4.10 in \cite{BBLeh1}, the 
capacity estimates in that paper hold for the exponent
sets defined above,
under appropriate restrictions of the radii.
We will use these facts without further ado.

The following example shows that it is possible to have
$\uqo < 1$, while as already mentioned 
we always have $\utho \ge 1$
provided that $X$ is connected, 
see Proposition~\ref{prop-uth-1} below.

\begin{example} \label{ex-power}
Let $X= \R^n$, $n \ge 2$, and $0< \alp < n$.
Then it is well known that $w(x)=|x|^{-\alp}$
is a Muckenhoupt $A_1$-weight and is thus 
$1$-admissible, by Theorem~4 in Bj\"orn~\cite{JBFennAnn}.
For $x_0=0$, it is easily verified that $\mu(B_r) \simeq r^{n-\al}$
and thus
\[
   \lQo=\lSo=(0,n-\al] 
\quad \text{and} \quad
   \uSo=\uQo=[n-\al,\infty).
\]
In particular, if $n-1 < \alp < n$, then 
$\lqo=\lso=\uso=\uqo=n-\alp<1$.
Moreover, 
$\ltho=n-\al$ and $\utho=n$.
\end{example}

For other examples with the exponent sets $\lQo$, $\lSo$, $\uSo$
and $\uQo$ having various properties,
see \cite[Section~3]{BBLeh1},
H.~Svensson~\cite{SvenssonH} and S.~Svensson~\cite{SvenssonS}.

\begin{prop} \label{prop-uth-1}
If $X$ is connected, then $\utho \ge 1$.
\end{prop}

\begin{proof}
Let $\uth \in \uTh$, $0<R<\frac{1}{4} \diam X$ and $x \in X$.
Then $X \setm B(x,2R)$ is nonempty.
As $X$ is connected, for each $0<\rho<2R$
there is $x_\rho$ such that $d(x,x_\rho)=\rho$.
Let $N\ge 2$ be an integer.

Then the balls $B^j:=B(x_{jR/N},R/2N)$, $j=1,\ldots,N-1$,
are pairwise disjoint and contained in $B(x,R)$,
and hence there is some $1
\le k \le N-1$ so that
\[
    \frac{\mu(B^{k})}{\mu(B(x,R))} \le  \frac{1}{N-1}.
\]
Thus, as $\uth \in \uTh$,
\[
    \frac{1}{N-1} 
    \ge \frac{\mu(B^{k})}  
{\mu(B(x,R))} 
    \ge  \frac{\mu(B^{k})}  
{\mu(B(x_{kR/N},2R))}  
    \ge \frac{1}{C} \biggl(\frac{1}{4N}\biggr)^\uth,
\]
where $C$ is the constant dictated by $\uth$.
As $N$ was
arbitrary this is possible only if $\uth \ge 1$.
Thus also  $\utho \ge 1$.
\end{proof}

\section{General capacity estimates for annuli}
\label{sect-meta}

Before going on to the core of this paper
-- the study of \p-harmonic Green functions --
we establish precise 
general estimates for the 
capacities of annuli.
These results will play 
an important role for instance in the pointwise estimates for
Green functions, see Theorem~\ref{thm-comp-u-R1-R2}.
Unlike in most of this paper, the estimates
in this section hold under rather weak assumptions.
We will consider the following pointwise properties.
Some of these conditions were introduced in \cite{BBLeh1}, but 
here it will be enough to
have them for certain radii.
Recall that $x_0 \in X$ is fixed.

\begin{deff}
We say that $\mu$ is \emph{doubling at $x_0$} if 
\[ 
  \mu(B(x_0,2r))\simle \mu(B(x_0,r))
\] 
for all radii $r>0$. 
Similarly, $\mu$ supports a \emph{\p-Poincar\'e inequality at $x_0$} 
if~\eqref{eq-deff-PI} holds for all balls $B=B(x_0,r)$.
Moreover, $\mu$ is \emph{reverse-doubling at $x_0$} if there
are constants $\xi, C>1$ 
such that 
\[ 
  \mu(B(x_0,\xi r))\ge C \mu(B(x_0,r))
\] 
for all $0<r\le \diam X/2\xi$.

We also say that a property, as above, holds at $x_0$ 
\emph{for radii up to $R_0$} if it holds for all $0< r \le R_0$.
(Here we allow for $R_0=\infty$, 
while $r$ is always finite, i.e.\ $r<R_0$ if $R_0=\infty$.)
Finally, a property
holds \emph{for small\/ \textup{(}resp.\ large\/\textup{)} radii},
if there is some $0<R_0<\infty$ such that the property
holds at $x_0$ for all $0< r \le R_0$ (resp.\ all $R_0 \le r <\infty$). 
\end{deff}

Note that if $X$ is bounded
and $\mu$ is reverse-doubling 
at $x_0$ for radii up to $R_0$, then necessarily $R_0 \le \diam X$.
(Letting $X=B(0,2) \subset \R^n$, equipped with
the metric $d(x,y)=\min\{|x-y|,1\}$ and the Lebesgue measure,
shows that it is possible to satisfy the reverse-doubling
condition with $R_0=\diam X$.)
It is easy to see by iteration
that if $\mu$ is doubling at $x_0$ for small radii, then $\uqo<\infty$,
and if $\mu$ is reverse-doubling at $x_0$ for small radii, then $\lqo>0$.

If 
$X$ is connected 
(which in particular holds 
if $\mu$ supports a global
Poincar\'e inequality)
and $\mu$ is globally
doubling, 
then $\mu$ is reverse-doubling
at every $x$ with constants $C>1$ and $\xi=2$ independent of $x$, 
see Lemma~3.7 in \cite{BBbook}
($\theta$ therein corresponds to $\xi$ here).

It is straightforward to show
that if $\mu$ is doubling or
reverse-doubling for large radii at one point $x_0$, then 
the same property holds at any other point, although the constants 
and radial bounds may change from point to point.
Similarly, if $\mu$ 
supports a \p-Poincar\'e inequality at $x_0$ for large radii, 
and $\mu$ is doubling at $x_0$ for large radii, 
then $\mu$ supports a \p-Poincar\'e inequality at any point for large radii.

Our main result in this section is the following estimate.

\begin{thm}\label{thm-metaestimate-int-only} 
Let $0 < R_0 \le \infty$.
Assume that 
\begin{enumerate}
\renewcommand{\theenumi}{\textup{(\roman{enumi})}}%
\item \label{a-i}
$\mu$ is reverse-doubling at $x_0$ for radii up to $R_0$,
with constant $\xi$,
\item
 $\mu$ is doubling at $x_0$ for radii up to $\max\bigl\{1,\tfrac12 \xi\bigr\}R_0$,
and
\item \label{a-iii}
$\mu$ supports a $p_0$-Poincar\'e inequality at $x_0$ for radii up to 
$\max\{2,\xi\}R_0$ for some $1\le p_0<p$.
\end{enumerate}
Then for all $0<2r\le R\le R_0$, 
\begin{equation}  \label{eq-lower-est-f}
\cp(B_r,B_R) 
\simeq \biggl( \int_r^R \biggl( \frac{\rho}{\mu(B_\rho)} \biggr)^{1/(p-1)} 
           \,d\rho \biggr)^{1-p},
\end{equation}
where the implicit comparison constants depend on $p$, $p_0$ and the
doubling, reverse-doubling
and Poincar\'e constants 
from \ref{a-i}--\ref{a-iii},
but not on $R_0$.
\end{thm}

As above, when $R_0=\infty$, we still require $R$ to be finite.

\begin{remark}  \label{rmk-metaestimate}
The proofs in this section reveal that 
the assumptions
in Theorem~\ref{thm-metaestimate-int-only}, 
Lemmas~\ref{lem:chain-estimate}
and~\ref{lem-meta-meta} and Corollary~\ref{cor-interpolate-cap}
about the (reverse) doubling and the Poincar\'e inequality 
can be further restricted to radii $\ge r$.
\end{remark}

Theorem~\ref{thm-metaestimate-int-only} 
gives an estimate of $\cp(B_r,B_R)$ for a large class of measures.
This generalizes many of the estimates in~\cite{BBLeh1},
which in turn were generalizations and improvements
of earlier results in Adamowicz--Shan\-mu\-ga\-lin\-gam~\cite{adsh}
and Garofalo--Marola~\cite{GaMa}.
On the other hand, the assumption of $p_0$-Poincar\'e inequality at $x_0$
for some $1\le p_0<p$ is stronger than in~\cite[Section~6]{BBLeh1}.

If $X$ is complete and $\mu$ is globally doubling and supports a 
global \p-Poincar\'e inequality, then 
by Keith--Zhong~\cite[Theorem~1.0.1]{KeZh} there is $1\le p_0<p$ such that
$X$ supports a global $p_0$-Poincar\'e inequality.
Under these assumptions $\mu$ is also reverse-doubling
with uniform constants $C>1$ and $\xi=2$,
and hence the capacity estimate in Theorem~\ref{thm-metaestimate-int-only}
holds with uniform constants for all $x_0\in X$ with 
$R_0=\frac14\diam X$.

Theorem~\ref{thm-metaestimate-int-only} can be used, for instance, 
to characterize when singletons have zero capacity
and when $X$ is \p-parabolic, see
Propositions~\ref{prop-meta-zero-cap} and~\ref{thm-p-parab}.
Some one-sided
estimates for capacities in terms of the 
volume growth, as in \eqref{eq-lower-est-f},
were given in
Coulhon--Holo\-pai\-nen--Saloff-Coste~\cite[pp.~1151 and~1162]{CoHoSC}, 
Holo\-pai\-nen~\cite[p.~329]{HoDuke} and
Holo\-pai\-nen--Koskela~\cite{HoKo}
mainly in the setting of Riemannian manifolds.

Other useful applications of Theorem~\ref{thm-metaestimate-int-only}
are the pointwise estimates for Green and singular functions
in Section~\ref{sect-ptwise-est}, as well as
the (non)integrability results for these functions and their
minimal \p-weak upper gradients in
Sections~\ref{sect-int-Green} and~\ref{sect-int-gu}.

\begin{example}
If $w(x)=w(|x|)$ is a radial weight on $\R^n$, $n\ge2$, such that the
measure $d\mu=w\,dx$ 
supports a \p-Poincar\'e inequality 
at $x_0=0$, then Proposition~10.8 in~\cite{BBLeh1} shows that
for all $0<r<R$,
\[
\cpmu(B_r,B_R) = \biggl( \int_r^R f'(\rho)^{1/(1-p)} \,d\rho \biggr)^{1-p},
\]
where $f(\rho)=\mu(B_\rho)$.
Thus, \eqref{eq-lower-est-f} 
can be seen as a generalization of this formula for annuli that are
not too thin.
Note that $f'(\rho)\simeq f(\rho)/\rho$ in many cases.
\end{example}

The following example shows that the assumption~\ref{a-iii}
of a Poincar\'e inequality cannot be dropped from
Theorem~\ref{thm-metaestimate-int-only}.

\begin{example}
Let $X=\itoverline{B(0,1)} \cup \{(x_1,x_2) : x_1 \ge 2\} \subset \R^2$
equipped with the Lebesgue measure $\mu$.
Then $X$ is complete and $\mu$ is globally doubling and globally reverse-doubling.
However, if $x_0=0$, $0 < r<2$ and $R >1$,
then $\cp(B_r,B_R)=0$.

A similar connected example is the bow-tie
\[
X=\{(x_1,x_2): |x_2| \le |x_1| \text{ and } x_1 \ge -1\} \subset \R^2
\]
equipped with the Lebesgue measure $\mu$.
Again $X$ is complete and $\mu$ is globally doubling and globally reverse-doubling.
However, if $x_0=(-1,0)$, $0 < r\le 1 < R $ and $1 < p \le 2$,
then $\cp(B_r,B_R)=0$. 
\end{example}

Lemma~2.6 in Heinonen--Kilpel\"ainen--Martio~\cite{HeKiMa}
(or Lemma~2.1 in Holo\-pai\-nen--Koskela~\cite{HoKo})
implies that for $R=2^{k_0}r$,
\begin{equation}   \label{eq-est-cap-HKM}
\cp(B_r,B_R) \le \biggl( 
\sum_{k=1}^{k_0} \cp(B^{k-1},B^{k})^{1/(1-p)} \biggr)^{1-p},
\end{equation}
where $B^k=2^k B_r$, $k=0,1,\ldots,k_0$. 
In fact, in~\cite{HeKiMa} and~\cite{HoKo}, \eqref{eq-est-cap-HKM}
is formulated for more general condensers, but
we are only interested in dyadic sequences of concentric balls.
The proof therein
only uses suitable convex combinations of test functions
and does not require any assumptions about the measure $\mu$.

Reformulating \eqref{eq-est-cap-HKM} 
as follows gives us the upper bound of 
Theorem~\ref{thm-metaestimate-int-only}.
The first inequality, with no doubling assumption, will be useful
when deducing Lemma~\ref{lem-suff-cap-zero},
while the second inequality is convenient when $\mu$ is doubling.

\begin{prop}  \label{prop-est-cap-int-upper}
For all $0<r\le\tfrac12R$, 
\begin{equation} \label{eq-upper-meta-a}
\cp(B_r,B_R)  \simle 
  \biggl( \int_{2r}^R \biggl( \frac{\rho}{\mu(B_\rho)} 
   \biggr)^{1/(p-1)} \,d\rho \biggr)^{1-p}
\end{equation}
and
\begin{equation}    \label{eq-upper-meta}
\cp(B_r,B_R)  \simle \frac{\mu(B_{2r})}{\mu(B_{r})}
\biggl( \int_r^R \biggl( \frac{\rho}{\mu(B_\rho)} 
   \biggr)^{1/(p-1)} \,d\rho \biggr)^{1-p}.
\end{equation}
\end{prop}

\begin{proof}
Write $r_k=2^k r$ and $B^k=B_{r_k}$, $k=0,1,\ldots$\,, 
and find an integer $k_0$ such that $r_{k_0} \le R < r_{k_0+1}$.
Proposition~5.1 in~\cite{BBLeh1} shows that 
$\cp(B^{k-1},B^{k}) \simle \mu(B^{k})/r_{k}^{p}$.
(Note that the proof of this part of Proposition~5.1 in~\cite{BBLeh1}
does not use any doubling property.)
Inserting this estimate into~\eqref{eq-est-cap-HKM} 
yields
\begin{align}
\cp(B_r,B_R) & \le \cp(B^0,B^{k_0}) 
   \label{eq-upper-meta-1}
\\
&\simle
\biggl( \sum_{k=1}^{k_0} \biggl(\frac{r_k^p}{\mu(B^k)} \biggr)^{1/(p-1)} 
\biggr)^{1-p}  
\simle\biggl( \int_{2r}^R \biggl( \frac{\rho}{\mu(B_\rho)} 
   \biggr)^{1/(p-1)} \,d\rho \biggr)^{1-p}. 
\nonumber
\end{align}
Finally,
\[
 \int_{r}^{2r} \biggl( \frac{\rho}{\mu(B_\rho)} 
   \biggr)^{1/(p-1)} \,d\rho 
   \le \biggl(\frac{2r}{\mu(B_r)} \biggr)^{1/(p-1)} r
   =\frac{1}{2} 
   \biggl( \frac{\mu(B_{2r})}{\mu(B_{r})}
           \frac{r_1^p}{\mu(B^1)} \biggr)^{1/(p-1)},
\]
which together with the last inequality in
\eqref{eq-upper-meta-1} gives
\eqref{eq-upper-meta}. 
\end{proof}

For the proof of the lower bound
in Theorem~\ref{thm-metaestimate-int-only}, 
we recall the following version of the well-known
``telescoping argument''.  

\begin{lem}\label{lem:chain-estimate}
\textup{(\cite[Lemma~4.9]{BBLeh1})}
Let $R_0 \in (0,\infty]$.
Assume that $1 \le p_0 < \infty$ and that
\begin{enumerate}
\renewcommand{\theenumi}{\textup{(\roman{enumi})}}%
\item \label{l-i}
$\mu$ is reverse-doubling at $x_0$ for radii up to $R_0$,
with constant $\xi$,
\item \label{l-ii}
 $\mu$ is doubling at $x_0$ for radii up to $\max\bigl\{1,\tfrac12 \xi\bigr\}R_0$,
and
\item \label{l-iii}
$\mu$ supports a $p_0$-Poincar\'e inequality at $x_0$ for radii up to 
$\max\{2,\xi\}R_0$.
\end{enumerate}
For $0 < 2r \le R_0$,
write $r_k=2^k r$ and $B^k=B_{r_k}$, $k=0,1,\ldots$\,, 
and let $k_0\ge 1$ 
 be such that $r_{k_0}\le R_0$.
Then we have for every $u\in \Npoo(B^{k_0})$ 
that
\[
|u_{B_r}| \simle \sum_{k=1}^{k_0} r_k 
     \biggl(\vint_{\lambda B^k} g_u^{p_0} \,d\mu\biggr)^{1/{p_0}},
\]
where $\la$ is the dilation constant in the $p_0$-Poincar\'e inequality at 
$x_0$.
\end{lem}

The assumptions in Lemma~\ref{lem:chain-estimate}
are slightly weaker
than in~\cite[Lemma~4.9]{BBLeh1}.
However, a careful check of the proof therein shows that 
only assumptions \ref{l-i}--\ref{l-iii}
are needed.
In particular, \ref{l-i} and \ref{l-ii}
are enough to guarantee the 
comparability of the measures in the second displayed formula
in the proof in~\cite{BBLeh1}, while
the Poincar\'e inequality is only used for the radii assumed
in \ref{l-iii}.

To make use of the above lemma 
we shall exploit the 
following general estimate that may be of independent interest since 
the assumption is very mild and
the first factor on the right-hand side of \eqref{eq-meta-meta} is strongly related
to the right-hand side of~\eqref{eq-est-cap-HKM}, cf.\ the proof of
Proposition~\ref{prop-est-cap-int-upper}.

\begin{lem}          \label{lem-meta-meta}
Let
$0 < R_0 \le R_0' \le \infty$
and assume that 
$\mu$ is reverse-doubling at $x_0$ for radii up to $R_0$.
For $0 < 2r \le R_0'$, 
write $r_k=2^k r$ and $B^k=B_{r_k}$, $k=0,1,\ldots$ and let $k_0\ge 1$
be such that $r_{k_0}\le R_0'$. 
Also let $1 \le p_0 < p$.
Then we have for every $g\in L^p(B^{k_0})$,
\begin{equation}   \label{eq-meta-meta}
\sum _{k=1}^{k_0} r_k \biggl( \vint_{B^k} g^{p_0} \,d\mu \biggr)^{1/p_0}
\simle  \biggl( \sum_{k=1}^{k_0} \biggl(\frac{r_k^p}{\mu(B^k)} \biggr)^{1/(p-1)} 
\biggr)^{1-1/p} \biggl( \int_{B^{k_0}} g^{p} \,d\mu \biggr)^{1/p},
\end{equation}
where the implicit comparison constant depends on $p$, $p_0$,
the reverse-doubling constants and $R_0'/R_0$.
\end{lem}

Before proving Lemma~\ref{lem-meta-meta} we
show how it leads to the lower bound in 
Theorem~\ref{thm-metaestimate-int-only}.

\begin{proof}[Proof of Theorem~\ref{thm-metaestimate-int-only}]
The upper bound follows from Proposition~\ref{prop-est-cap-int-upper}.

Conversely, let $0\le u\in\Np(X)$ be admissible for
$\cp(B_r,B_R)$.
Write $r_k=2^k r$ and $B^k=B_{r_k}$, $k=0,1,\ldots$\,, 
and find an integer $k_0$ such that $r_{k_0-1} < R \le r_{k_0}$.
The telescoping Lemma~\ref{lem:chain-estimate}, 
followed by
Lemma~\ref{lem-meta-meta} applied to the balls $\la B^k$ in place of $B^k$
and $R_0'=\la R_0$,
gives
\[ 
1 \simle 
\sum _{k=1}^{k_0} r_k \biggl( \vint_{\la B^k} g_u^{p_0} \,d\mu \biggr)^{1/p_0}
\simle  
  \biggl( \sum_{k=1}^{k_0} \biggl(\frac{r_k^p}{\mu(B^k)} \biggr)^{1/(p-1)} 
    \biggr)^{1-1/p} \biggl( \int_{B_R} g_u^{p} \,d\mu \biggr)^{1/p}. 
\] 
Taking infimum over all $u$ admissible for
$\cp(B_r,B_R)$
and replacing the sum on
the right-hand side by the corresponding
integral yields the lower bound in~\eqref{eq-lower-est-f}.
\end{proof}

\begin{remark}   \label{rem-lower-cap}
Under the assumptions in Theorem~\ref{thm-metaestimate-int-only}, by
Propositions~5.1 and~6.2 
in~\cite{BBLeh1} we have
\[
\cp(B^{k-1},B^{k}) \simeq \frac{\mu(B^{k})}{r_{k}^p}
\]
and hence the proof 
of Theorem~\ref{thm-metaestimate-int-only} shows that the lower bound
in~\eqref{eq-lower-est-f}
can also be written as
\[ 
\cp(B_r,B_R) \simge
\biggl( \sum _{k=1}^{k_0} \cp(B^{k-1},B^k)^{1/(1-p)} \biggr)^{1-p},
\] 
where $R=2^{k_0}r$. 
Thus, \eqref{eq-est-cap-HKM} is essentially sharp.
\end{remark}

\begin{proof}[Proof of Lemma~\ref{lem-meta-meta}]
Write the left-hand side $S$ in \eqref{eq-meta-meta} as
\begin{equation}   \label{eq-telescopic-estp0}
S = \sum _{k=1}^{k_0}\frac{r_k}{\mu(B^k)^{1/p_0}} 
          \biggl( \int_{B^k} g^{p_0} \,d\mu \biggr)^{1/p_0}.
\end{equation}
We split the integral over $B^k$ into integrals over $A_0:=B^0$
and the annuli 
$A_j=B^j\setm B^{j-1}$, $j=1,2,\ldots,k$, apply H\"older's 
inequality (with $p/p_0$ and $p_1:=p/(p-p_0)$) to each of these 
integrals, and obtain 
\begin{equation}    \label{eq-Holder-int-Aj}
\int_{B^k} g^{p_0} \,d\mu 
\le \sum_{j=0}^k \biggl( \int_{A_j} g^{p} \,d\mu \biggr)^{p_0/p} 
           \mu(A_j)^{1/p_1}.
\end{equation}
The reverse-doubling property at $x_0$ implies that for some $\be>0$,
we can estimate $\mu(A_j)^{1/p_1}$ as
\[
\mu(A_j)^{1/p_1} \le \mu(B^j)^{1/p_1} \simle \mu(B^k)^{1/p_1} 
        \Bigl( \frac{r_j}{r_k} \Bigr)^{2\be},
\]
where the comparison constant depends on $R_0'/R_0$.
Since $1/p_1=1-p_0/p$, inserting this first into~\eqref{eq-Holder-int-Aj}
and then into~\eqref{eq-telescopic-estp0}  gives
\begin{equation}   \label{eq-using-q-in-uQ}
S \simle \sum_{k=1}^{k_0} \frac{r_k}{\mu(B^k)^{1/p_0}} 
     \frac{\mu(B^k)^{1/p_0-1/p}} {r_k^{2\be/p_0}} 
   \biggl( \sum_{j=0}^k r_j^{2\be} 
         \biggl( \int_{A_j} g^{p} \,d\mu \biggr)^{p_0/p} \biggr)^{1/p_0}.
\end{equation}
The last sum in~\eqref{eq-using-q-in-uQ} is now estimated using
H\"older's inequality for sums 
(with $p_1=p/(p-p_0)$ and $p/p_0$ again) 
as follows
\begin{align*}
\sum_{j=0}^k r_j^{\be} r_j^\be
         \biggl( \int_{A_j} g^{p} \,d\mu \biggr)^{p_0/p}
&\le \biggl( \sum_{j=0}^k r_j^{\be p_1} \biggr)^{1/p_1}
     \biggl( \sum_{j=0}^k r_j^{\be p/p_0} 
          \int_{A_j} g^{p} \,d\mu \biggr)^{p_0/p} \\
&\simeq r_k^{\be} \biggl( \sum_{j=0}^k r_j^{\be p/p_0} 
          \int_{A_j} g^{p} \,d\mu \biggr)^{p_0/p}.
\end{align*}
Inserting this into~\eqref{eq-using-q-in-uQ} yields
\[
S \simle \sum _{k=1}^{k_0} \biggl(\frac{r_k^{p}}{\mu(B^k)} \biggr)^{1/p} 
  r_k^{-\be/p_0} \biggl( \sum_{j=0}^k r_j^{\be p/p_0} 
          \int_{A_j} g^{p} \,d\mu \biggr)^{1/p}.
\]
Another use of H\"older's inequality for sums 
(with $p/(p-1)$ and $p$)
implies
\[
S \simle \biggl( \sum _{k=1}^{k_0} \biggl(\frac{r_k^{p}}{\mu(B^k)} 
              \biggr)^{1/(p-1)} \biggr)^{1-1/p}
  \biggl( \sum _{k=1}^{k_0} r_k^{-\be p/p_0} \sum_{j=0}^k r_j^{\be p/p_0} 
          \int_{A_j} g^{p} \,d\mu \biggr)^{1/p}.
\]
The last factor is estimated by changing the order of summation as
\[
\biggl( \sum_{j=0}^{k_0} r_j^{\be p/p_0} \int_{A_j} g^{p}  \,d\mu 
           \sum_{k=\max\{1,j\}}^{k_0} r_k^{-\be p/p_0} \biggr)^{1/p}
\simeq \biggl( \sum_{j=0}^{k_0} \int_{A_j} g^{p} \,d\mu \biggr)^{1/p},
\]
since the geometric sum is comparable to $r_j^{-\be p/p_0}$.
We can thus conclude that
\[ 
S \simle 
\biggl( \sum_{k=1}^{k_0} \biggl(\frac{r_k^p}{\mu(B^k)} \biggr)^{1/(p-1)} 
\biggr)^{1-1/p} \biggl( \sum _{j=0}^{k_0} \int_{A_j} g^{p} \,d\mu \biggr)^{1/p}, 
\] 
and the claim follows.
\end{proof}

As a consequence of Theorem~\ref{thm-metaestimate-int-only} we
obtain the following estimates
for different capacities, which will be important
when deducing Theorem~\ref{thm-nonintegrability-gu}.

\begin{cor}  \label{cor-interpolate-cap}
Let $1<q<t<p$, $\al=(p-t)/(p-q)$
and $R_0 \in (0,\infty]$.
Assume that 
\begin{enumerate}
\renewcommand{\theenumi}{\textup{(\roman{enumi})}}%
\item \label{b-i}
$\mu$ is reverse-doubling at $x_0$ for radii up to $R_0$,
with constant $\xi$,
\item
 $\mu$ is doubling at $x_0$ for radii up to $\max\bigl\{1,\tfrac12 \xi\bigr\}R_0$,
and
\item \label{b-iii}
$\mu$ supports a $t_0$-Poincar\'e inequality at $x_0$ for radii up to 
$\max\{2,\xi\}R_0$ for some\/ $1 \le t_0 <t$.
\end{enumerate}
Let $0<2r\le R\le R_0$.
Then the following are true, 
with the implicit comparison constants depending on 
$p$, $q$, $t$, 
the {\rm(}reverse\/{\rm)} doubling and Poincar\'e constants 
from \ref{b-i}--\ref{b-iii},
and in \ref{int-b} and \ref{int-c} also on $R_0$.

\begin{enumerate}
\item \label{int-a}
In general,
\begin{align*}   
  \capp_t(B_r,B_R) &\simge  \capp_p(B_r,B_R)^{1-\al}
  \biggl( \int_r^R \biggl( \frac{\rho}{\mu(B_\rho)} \biggr)^{1/(q-1)} 
           \,d\rho \biggr)^{\alp(1-q)} \\
 &\simge \capp_p(B_r,B_R)^{1-\al}\capp_q(B_r,B_R)^\al.
\end{align*}
\item \label{int-b}
If $q < \lqo$ and $R_0 <\infty$, 
then
\[ 
  \capp_t(B_r,B_R) \simge  \capp_p(B_r,B_R)^{1-\al}
  \biggl( \frac{\mu(B_r)}{r^q}\biggr)^{\alp}.
\] 

\item \label{int-c}
If $q = \lqo \in \lQo$ and $R_0 <\infty$, 
then
\[ 
  \capp_t(B_r,B_R) \simge  \capp_p(B_r,B_R)^{1-\al}
  \biggl( \frac{\mu(B_r)}{r^q}\biggr)^{\alp}
  \biggl( \log \frac{R}{r}\biggr)^{\alp(1-q)}.
\] 
\end{enumerate}
\medskip
\end{cor}

Note that since $p>t$, \eqref{eq-lower-est-f} holds for $\capp_p$,
which can therefore be replaced  in 
Corollary~\ref{cor-interpolate-cap} by a corresponding integral.
Also observe that \ref{int-b} and \ref{int-c} do not
follow from \ref{int-a} together with the estimates
in \cite{BBLeh1} since we here assume a weaker
Poincar\'e inequality.

\begin{proof}
\ref{int-a}
Let $\be = 1-\al = (t-q)/(p-q)$ and note that
\[
\frac{\al(q-1)}{t-1} + \frac{\be(p-1)}{t-1} =1
\quad \text{and} \quad
\al q + \be p = t.
\]
H\"older's inequality then implies that
\begin{align*}
\int_r^R \biggl( \frac{\rho}{\mu(B_\rho)} \biggr)^{1/(t-1)} \,d\rho 
&= \int_r^R \biggl( \frac{\rho}{\mu(B_\rho)} \biggr)^{\al/(t-1)+\be/(t-1)} \,d\rho \\
& \le \biggl( \int_r^R \biggl( \frac{\rho}{\mu(B_\rho)} \biggr)^{1/(q-1)} \,d\rho 
                 \biggr)^{\al(q-1)/(t-1)} \nonumber \\
& \quad  \times
       \biggl( \int_r^R \biggl( \frac{\rho}{\mu(B_\rho)} \biggr)^{1/(p-1)} \,d\rho 
                 \biggr)^{\be(p-1)/(t-1)}.  \nonumber
\end{align*}
Raising both sides to the power $1-t$ and applying 
Theorem~\ref{thm-metaestimate-int-only} to
$\capp_t$, and \eqref{eq-upper-meta} to $\capp_p$ and $\capp_q$,
now yields~\ref{int-a}.

\ref{int-b}
Let $q<q'<\lqo$ and $\ga=(q'-1)/(q-1)>1$.
Then
\[
    \frac{\rho^{q'}}{\mu(B_\rho)}
    \simle  \frac{r^{q'}}{\mu(B_r)}
    \quad \text{for } r < \rho \le R_0,
\]
and so
\begin{align*}
  \int_r^R \biggl( \frac{\rho}{\mu(B_\rho)} \biggr)^{1/(q-1)} \,d\rho
   & = 
  \int_r^R \biggl( \frac{\rho^{q'}}{\mu(B_\rho)} \biggr)^{1/(q-1)}\,
   \frac{d\rho}{\rho^\ga} \\ 
   & \simle 
  \int_r^R \biggl( \frac{r^{q'}}{\mu(B_r)} \biggr)^{1/(q-1)}\,
   \frac{d\rho}{\rho^\ga}  
    \simle 
  \biggl( \frac{r^{q'}}{\mu(B_r)} \biggr)^{1/(q-1)} r^{1-\ga}.
\end{align*}
Hence,  
\[
\biggl(\int_r^R \biggl( \frac{\rho}{\mu(B_\rho)} \biggr)^{1/(q-1)} \,d\rho\biggr)^{1-q}
\simge \frac{\mu(B_r)}{r^{q'}} r^{(\ga-1)(q-1)} 
=  \frac{\mu(B_r)}{r^{q}},
\]
and inserting this into \ref{int-a} gives \ref{int-b}.

\ref{int-c}
Proceeding as in \ref{int-b}, with $q'=q$, we see that 
\[ 
  \int_r^R \biggl( \frac{\rho}{\mu(B_\rho)} \biggr)^{1/(q-1)} \,d\rho
    \simle \biggl( \frac{r^{q}}{\mu(B_r)} \biggr)^{1/(q-1)}
   \int_r^R \frac{d\rho}{\rho}
   = \biggl( \frac{r^{q}}{\mu(B_r)} \biggr)^{1/(q-1)} \log \frac{R}{r}.
\] 
Inserting this into \ref{int-a} gives \ref{int-c}.
\end{proof}

The dependence of the implicit constants in \ref{int-b} 
and~\ref{int-c} on $R_0$ 
is through the constant $C_q$ appearing
in the definition of $\lQo$ for $0<r<R\le R_0$.
It therefore also depends on the particular choice of $q'>q$ in the
proof of \ref{int-b}.
If $R_0=\infty$, 
then $\diam X = \infty$ and
\ref{int-b} and \ref{int-c} hold for 
$q < \lqq :=\sup \lQ$ 
and $q = \lqq \in \lQ$, respectively, where
\[
\lQ  =\biggl\{q>0 : 
        \frac{\mu(B_r)}{\mu(B_R)}  \simle
        \Bigl(\frac{r}{R}\Bigr)^q 
        \text{ for all } 0 < r < R< \infty 
\biggr\}.
\]

\begin{remark}
Choosing $q<\min\{t,\lqo\}$ and $p>\max\{t,\uqo\}$ 
in Corollary~\ref{cor-interpolate-cap}\,\ref{int-b},
together with Proposition~6.1\,(b) 
in~\cite{BBLeh1} and a direct calculation,
yields for $0<2r<R\le \diam X/2\xi$ that
\begin{equation} \label{eq-approx-cap}
\capp_t(B_r,B_R) \simge \biggl( \frac{\mu(B_r)}{r^t} \biggr)^{(p-t)/(p-q)}
    \biggl( \frac{\mu(B_R)}{R^t} \biggr)^{(t-q)/(p-q)}
\Bigl( \frac{r}{R} \Bigr)^{(p-t)(t-q)/(p-q)}.
\end{equation}
Since
\[
\frac{\mu(B_r)}{r^t} \simge \capp_t(B_r,B_R) \quad \text{and} \quad
\frac{\mu(B_R)}{R^t} \simge \capp_t(B_r,B_R),
\]
this implies and improves the estimates in~\cite[Proposition~6.2]{BBLeh1},
which use only one of the balls $B_r$ and $B_R$.
The borderline cases $q=\lqo \in \lQo$ and $p=\uqo \in \uQo$
which are allowed in \cite[Proposition~6.2]{BBLeh1}, are however not
included in \eqref{eq-approx-cap},
and the Poincar\'e assumption is slightly stronger here.

Note that the product of the estimates in (a) and (b) 
of~\cite[Proposition~6.2]{BBLeh1} gives an estimate similar to 
\eqref{eq-approx-cap}, but with twice as large exponent at $r/R$.
Moreover, \cite[Proposition~5.1]{BBLeh1} implies that
\[
\capp_t(B_r,B_R) \simle \biggl( \frac{\mu(B_r)}{r^t} \biggr)^{(p-t)/(p-q)}
    \biggl( \frac{\mu(B_R)}{R^t} \biggr)^{(t-q)/(p-q)}.
\]
If $q=\lqo \in \lQo$ and $p=\uqo \in \uQo$,
one can combine Corollary~\ref{cor-interpolate-cap}
with~\cite[ Proposition~7.1]{BBLeh1}
to obtain more explicit lower
bounds for $\capp_t$, also containing $\log(R/r)$.
\end{remark}

\section{Capacity of singletons and \texorpdfstring{\p}{p}-parabolicity}
\label{sect-cap-zero}

By letting $r\to0$ or $R\to\infty$ in
Theorem~\ref{thm-metaestimate-int-only},
we will in this section characterize points of zero capacity and
\p-parabolic metric spaces in terms of integrals of the type
\eqref{eq-lower-est-f}.
This gives more precise descriptions than some earlier conditions
based on dimensions and exponent sets, as in
the following result from~\cite{BBLeh1}.

\begin{prop} \label{prop-cp-x0}
\textup{(Proposition~8.2 in~\cite{BBLeh1})}
\begin{enumerate}
\item \label{e-a}
If $p<\uso$ or $p=\uso\notin \uSo\setm\lSo$, then $\Cp(\{x_0\})=0$.
\item \label{e-c}
If $p>\uso$, $\mu$ is doubling and reverse-doubling at $x_0$ and
supports a \p-Poincar\'e inequality at $x_0$, all three properties
holding for small radii,
and
$x_0$ has a locally compact neighbourhood,
then $\Cp(\{x_0\})>0$.
\end{enumerate}
\end{prop}

A careful check of the proofs in \cite{BBLeh1} 
shows that no doubling assumption is needed for 
Proposition~\ref{prop-cp-x0}\,\ref{e-a}.
Already Holo\-pai\-nen--Shanmugalingam~\cite{HoSha},
in the comment following the proof of Lemma~3.6 therein,
pointed out
that if $p \in \lSo$ and $X$ is locally compact, then
$\cp(\{x_0\},\Om)=0$ whenever $\Om \ni x_0$ is open,
from which it easily follows that $\Cp(\{x_0\})=0$
(cf.\ \eqref{eq-cp=>Cp} below).

If $p=\uso \in \uSo \setm \lSo$, then
the exponent sets are not fine enough to 
capture when $x_0$ has zero capacity, see Example~9.4 in \cite{BBLeh1}. 
The following results, 
which are based on the
general capacity estimates from Section~\ref{sect-meta},
are therefore of interest.

\begin{lem}\label{lem-suff-cap-zero}
Let $\Om\subset X$ be a bounded open set with $x_0\in\Om$.
If 
 \begin{equation}   \label{eq-cap-0-int-infty}
 \int_0^{\delta} \biggl( \frac{\rho}{\mu(B_\rho)} \biggr)^{1/(p-1)} \,d\rho = \infty
 \quad\text{for some $\de>0$,}
 \end{equation}
or equivalently
 \begin{equation}   \label{eq-cap-0-int-infty-sum}
 \sum_{k=k_0}^\infty \biggl( \frac{2^{-kp}}{\mu(B_{2^{-k}})} \biggr)^{1/(p-1)}  = \infty
 \quad\text{for some integer $k_0\ge 0$,}
 \end{equation}
then $\cp(\{x_0\},\Omega)=\Cp(\{x_0\})=0$.
\end{lem}

Note that the conditions in~\eqref{eq-cap-0-int-infty} 
and \eqref{eq-cap-0-int-infty-sum} can be equivalently
required
for all $\de>0$ and all integers $k_0\ge 0$, respectively.

\begin{proof}
We may assume that $B(x_0,\de)\subset\Om$.
Hence it follows from 
\eqref{eq-upper-meta-a}
and \eqref{eq-cap-0-int-infty} that
\[
\cp(\{x_0\},\Omega)\le \lim_{r \to 0} \cp(B_r,B_\de) \simle 
\lim_{r \to 0} \biggl( \int_{2r}^\de \biggl( \frac{\rho}{\mu(B_\rho)} \biggr)^{1/(p-1)} 
           \,d\rho \biggr)^{1-p} = 0.
\]
For the second part,
we have $\mu(\{x_0\})=0$ since the integral in~\eqref{eq-cap-0-int-infty}
diverges, and thus
\begin{equation} \label{eq-cp=>Cp}
    \Cp(\{x_0\}) 
   \le \lim_{\de \to 0} \bigl(\mu(B(x_0,\de))+   \cp(\{x_0\},B(x_0,\de))\bigr)
   = \lim_{\de \to 0} \mu(B(x_0,\de))
   =0,
\end{equation}
by the regularity of the Borel regular measure $\mu$.
\end{proof}

To obtain also the converse direction, we need stronger (pointwise)
assumptions as follows.

\begin{prop}\label{prop-meta-zero-cap}
Assume that 
\begin{enumerate}
\renewcommand{\theenumi}{\textup{(\roman{enumi})}}%
\item 
$\mu$ is reverse-doubling at $x_0$ for small radii,
\item
 $\mu$ is doubling at $x_0$ for small radii,
and
\item \label{kk-iii}
$\mu$ supports a $p_0$-Poincar\'e inequality at $x_0$ for small radii 
and some $1\le p_0<p$.
\end{enumerate}
Let $\Om\subset X$ be a bounded open set with $x_0 \in \Om$,
and assume that $x_0$ has a 
locally
compact neighbourhood.
\begin{enumerate}
\item \label{prop-meta-zero-cap-a}
Then $\Cp(\{x_0\})=0$ if and only if~\eqref{eq-cap-0-int-infty} holds. 
\item \label{prop-meta-zero-cap-c}
If 
$\mu$ supports a \p-Poincar\'e inequality at $x_0$,
then
$\cp(\{x_0\},\Omega)=0$ if and only if either~\eqref{eq-cap-0-int-infty} holds
or $\Cp(X \setm \Om)=0$.
\end{enumerate}
\end{prop}

Instead of~\eqref{eq-cap-0-int-infty} one can  
equivalently require 
that~\eqref{eq-cap-0-int-infty-sum} holds.
For $n \ge 2$, consider the following union of concentric layers
\[
   X=\{ x \in \R^{n} : |x| \notin E\}, 
\quad \text{where } E=\bigcup_{j=1}^\infty (2^{-2j},2^{1-2j}) \subset \R,
\]
equipped with the Lebesgue measure.
In this case $\Cp(\{0\})=0$ for all $1<p<\infty$, but \eqref{eq-cap-0-int-infty}
fails if $p>n$ and $x_0=0$.
Thus the assumption~\ref{kk-iii} cannot be dropped, even if $\mu$
is assumed to be globally doubling and globally reverse-doubling.
That the extra Poincar\'e assumption in \ref{prop-meta-zero-cap-c} cannot be 
dropped 
is easily seen by considering e.g.\
$X=\itoverline{B(0,1)} \cup \itoverline{B(3,1)}$ in $\R^n$, with $x_0=0$,
$\Om = \itoverline{B(0,1)}$ and $p>n$.

\begin{proof}
\ref{prop-meta-zero-cap-a}
Assume first that $\Cp(\{x_0\})=0$, and let $\eps,\de>0$.
By Proposition~4.7 in~\cite{BBLeh1},
there is $0<r<\de$ such that $\cp(B_r,B_\de)<\eps$. 
Since $\eps>0$ was arbitrary, 
we obtain from Theorem~\ref{thm-metaestimate-int-only}
that~\eqref{eq-cap-0-int-infty} holds.
The converse implication follows directly from Lemma~\ref{lem-suff-cap-zero}.

\ref{prop-meta-zero-cap-c}
If $\cp(\{x_0\},\Omega)=0$ and $\Cp(X \setm \Om)>0$,
then 
\cite[Proposition~4.6]{BBLeh1} shows that
$\Cp(\{x_0\})=0$, and thus~\eqref{eq-cap-0-int-infty} holds by 
part~\ref{prop-meta-zero-cap-a}.
The converse implication follows from 
Lemma~\ref{lem-suff-cap-zero} and the fact that if
$\Cp(X \setm \Om)=0$ then $u \equiv 1$ is admissible in  
the definition of $\cp(\{x_0\},\Omega)$, which is thus zero.
\end{proof}

\begin{deff} \label{def-p-par}
An unbounded space $X$ is \emph{\p-parabolic} if 
$\cp(B,X)=0$ for all balls $B\subset X$,
otherwise it is \emph{\p-hyperbolic}.
\end{deff}

On (sub)Riemannian manifolds, \p-parabolicity is often defined 
as the \emph{nonexistence of global \p-harmonic Green functions}, 
which in those situations is known to be equivalent to the above
requirement that $\cp(B,X)=0$ for all balls $B\subset X$.
In the generality of this section, there is no  available
theory for \p-harmonic functions.
Even in the standard setting of complete metric spaces 
with a doubling measure supporting a Poincar\'e inequality, 
global \p-harmonic Green functions are little studied.

In Holo\-pai\-nen~\cite[p.\ 322]{HoDuke}, 
Holo\-pai\-nen--Koskela~\cite[p.\ 3428]{HoKo} and 
Holo\-pai\-nen--Shanmugalingam~\cite[Definition~3.13]{HoSha},
\p-parabolicity was defined by requiring that $\cp(K,X)=0$
for all compact sets $K$.
This is equivalent to Definition~\ref{def-p-par} provided that $X$ is proper,
but in nonproper spaces our definition seems to be more relevant.

Sufficient and/or necessary conditions for \p-parabolicity 
using the integral \eqref{eq-int-parab=infty} below
have been obtained under various assumptions in a number of papers, 
see e.g.\ \cite{CoHoSC}, \cite{Ho}--\cite{HoSha}
and the end of the introduction for more details.
In \cite[Proposition~8.6 and Remark~8.7]{BBLeh1} we gave simple
conditions for $\cp(B,X)=0$ (and thus 
\p-parabolicity) in terms of exponent sets defined for large radii
similarly to $\lSo$ and $\uSo$.
As in Proposition~\ref{prop-cp-x0}, 
also here there are cases in which
these exponent sets
are not fine enough to
capture when \p-parabolicity holds,
but using the estimate in Theorem~\ref{thm-metaestimate-int-only}
we are now able to give a precise characterization,
under mild assumptions.

\begin{thm}   \label{thm-p-parab}
Let $X$ be unbounded and 
$r_0>0$.
Then $X$ is \p-parabolic 
if 
\begin{equation}   \label{eq-int-parab=infty}
\int_{r_0}^\infty \biggl( \frac{\rho}{\mu(B_\rho)} \biggr)^{1/(p-1)} \,d\rho=\infty.
\end{equation}

If moreover  $\mu$ is doubling and reverse-doubling at $x_0$ and supports
a $p_0$-Poincar\'e inequality at $x_0$ for some $1\le p_0<p$,
all three properties holding for large radii,
then $X$ is \p-parabolic 
if and only if \eqref{eq-int-parab=infty} holds.
\end{thm}

Condition~\eqref{eq-int-parab=infty} is clearly
independent of the choices of $r_0$ and $x_0$.
The first part recovers
Proposition~2.3 in Holo\-pai\-nen--Koskela~\cite{HoKo}.

\begin{proof}
For the first part, let
$B\subset X$ be a ball. Then $B\subset B_r$ for some $r\ge r_0$,
and for $R\ge 2r$ we have by 
Proposition~\ref{prop-est-cap-int-upper}
that
\[
\cp(B,X) \le \cp(B_r,B_R) 
\simle \biggl( \int_{2r}^R \biggl( \frac{\rho}{\mu(B_\rho)} \biggr)^{1/(p-1)} 
          \,d\rho \biggr)^{1-p} \to 0, \quad \text{as }R\to\infty,
\]
whenever \eqref{eq-int-parab=infty} holds.

For the converse implication in the second part we may assume that
the doubling and reverse-doubling conditions and
the Poincar\'e 
inequality hold for balls $B_r=B(x_0,r)$ with $r \ge r_0$.
If $\cp(B_r,X)=0$ for some $r\ge r_0$, then let $\eps>0$
be arbitrary and find $u\in \Np(X)$ so that  $u\ge1$ on $B_r$ and
\(
\int_X g_u^p \,d\mu <\eps.
\)
For $R\ge 2r+1$, let $\eta(x)=\min\{(R-d(x,x_0))_\limplus,1\}$
be a cut-off function.
Then $u\eta$ is admissible for $\cp(B_r,B_R)$ and hence,
by Theorem~\ref{thm-metaestimate-int-only} and Remark~\ref{rmk-metaestimate},
\begin{align*}
&\biggl( \int_{r}^R \biggl( \frac{\rho}{\mu(B_\rho)} \biggr)^{1/(p-1)} 
          \,d\rho \biggr)^{1-p}
\simeq \cp(B_r,B_R) \le \int_X g_{u\eta}^p \,d\mu \\
&\quad \quad \quad \le 2^p \biggl( \int_X g_{u}^p \,d\mu 
    + \int_{B_R\setm B_{R-1}} |u|^p \,d\mu \biggr) 
     \to 2^p \int_X g_{u}^p \,d\mu  < 2^p\eps,
\end{align*}
as $R\to\infty$, since $u\in L^p(X)$.
As $\eps$ was arbitrary, we conclude that \eqref{eq-int-parab=infty}
holds ($r_0$ therein can clearly be replaced by $r$).
\end{proof}

The following example shows that the assumption
of a Poincar\'e inequality cannot be dropped from
the second part of
Theorem~\ref{thm-p-parab}.

\begin{example}
For $n \ge 2$, consider the following 
union of concentric layers
\[
   X=\{ x \in \R^{n} : |x| \notin E\}, 
\quad \text{where } E=\bigcup_{j=1}^\infty (2^{2j-1},2^{2j}) \subset \R,
\]
equipped with the Lebesgue measure.
In this case $X$ is \p-parabolic for all $1<p<\infty$, 
but \eqref{eq-int-parab=infty}
fails if $1<p <n$. 
Thus the Poincar\'e assumption cannot be dropped, even if $\mu$
is assumed to be globally doubling and globally reverse-doubling.

Similar connected examples are given by ``the infinite chessboard''
\[
    X=\R^{2} \setm \bigcup_{j,k=-\infty}^\infty \bigl((2j,2j+1) \times (2k,2k+1)\bigr)
    \cup \bigl((2j-1,2j) \times (2k-1,2k)\bigr)
\]
and its generalizations to $\R^n$, $n \ge 3$.
\end{example}

\section{\texorpdfstring{\p}{p}-harmonic,
singular and Green functions}
\label{sect-harm}

\emph{From now on, we assume that $X$ is complete, that
$\mu$ is doubling and supports a \p-Poincar\'e inequality,
and that $\Om \subset X$ 
is a nonempty open set with $x_0 \in \Om$.
As always in this paper, $1<p<\infty$.}

\medskip

In this section we first recall the definitions of
\p-harmonic and superharmonic functions 
and present some of their important properties that
will be needed later.
After that we recall 
results from Bj\"orn--Bj\"orn--Lehrb\"ack~\cite{BBLehGreen}
on the existence and properties of
singular and Green functions.

It follows from the assumptions that
$X$ is proper and quasiconvex 
and thus also connected and locally connected.
Moreover, 
$\mu$ is reverse-doubling with $\xi=2$.
These facts will be important to keep in mind.
Recall also that by Keith--Zhong~\cite[Theorem~1.0.1]{KeZh}, 
$X$ supports a $p_0$-Poincar\'e inequality
for some $1 \le p_0<p$. 
This is assumed explicitly in some of the papers we refer to below.

The results in the rest of the paper
also hold if $X$ is a proper metric space equipped
with a locally doubling measure $\mu$ that
supports a local \p-Poincar\'e inequality,
as defined in Bj\"orn--Bj\"orn~\cite{BBsemilocal},
see \cite[Section~11]{BBLehGreen}.
The dependence on the constants will then
be affected in a natural way.

\begin{deff} \label{def-quasimin}
A function $u \in \Nploc(\Om)$ is a
\emph{\textup{(}super\/\textup{)}minimizer} in $\Om$
if 
\[ 
      \int_{\phi \ne 0} g^p_u \, d\mu
           \le \int_{\phi \ne 0} g_{u+\phi}^p \, d\mu
           \quad \text{for all (nonnegative) } \phi \in \Np_0(\Om).
\] 
A \emph{\p-harmonic function} is a continuous minimizer
(by which we mean real-valued continuous in this paper).
\end{deff}

For various characterizations of minimizers and superminimizers  
see Bj\"orn~\cite{ABkellogg}.
It was shown in Kinnunen--Shanmugalingam~\cite{KiSh01} that 
under our standing
assumptions,
a minimizer can be modified on a set of zero (Sobolev) capacity to obtain
a \p-harmonic function. For a superminimizer $u$,  it was shown by
Kinnunen--Martio~\cite{KiMa02} that its \emph{lsc-regularization}
\[ 
 u^*(x):=\essliminf_{y\to x} u(y)= \lim_{r \to 0} \essinf_{B(x,r)} u
\] 
is also a superminimizer  and $u^*= u$ q.e.

\begin{deff} \label{deff-superharm}
A function $u \colon \Om \to (-\infty,\infty]$ is 
\emph{superharmonic}
in $\Om$ if
\begin{enumerate}
\renewcommand{\theenumi}{\textup{(\roman{enumi})}}%
\item \label{cond-a} $u$ is lower semicontinuous;
\item \label{cond-b} 
 $u$ is not identically $\infty$ in any component of $\Om$;
\item \label{cond-c}
for every nonempty open set $G \Subset \Om$ with $\Cp(X \setm G)>0$,
and all 
functions $v \in C(\itoverline{G})$ such that
$v$ is \p-harmonic in $G$
we have $v \le u$ in $G$
whenever $v \le u$ on $\bdy G$.
\end{enumerate}
\end{deff}

As usual, by $G \Subset \Om$ we mean that $\itoverline{G}$
is a compact subset of $\Om$.
By Theorem~6.1 in Bj\"orn~\cite{ABsuper}
(or \cite[Theorem~14.10]{BBbook}),
this definition of superharmonicity is equivalent to  the 
definition usually used
in metric spaces, e.g.\ in \cite{BBbook} and \cite{BBLehGreen}.
It also coincides with the definitions used
in $\R^n$ and on Riemannian manifolds.
Superharmonic functions 
are always lsc-regularized (i.e.\ $u^*=u$).
Any lsc-regularized superminimizer is superharmonic, and
conversely any \emph{bounded} superharmonic function is an lsc-regularized superminimizer
and thus belongs to $\Nploc(\Om)$.

The following definition of singular and Green functions
in metric spaces was given in~\cite[Definition~1.1]{BBLehGreen}.
Recall that a \emph{domain} is a nonempty open connected set.

\begin{deff} \label{deff-sing}
Let $\Om \subset X$ be a bounded domain. 
A positive
function $u\colon\Om \to (0,\infty] $ is a \emph{singular function} in $\Om$
with singularity at $x_0 \in \Om$ if it satisfies the following
properties\/\textup{:} 
\begin{enumerate}
\renewcommand{\theenumi}{\textup{(S\arabic{enumi})}}%
\item \label{dd-s}
  $u$ is superharmonic
  in $\Om$\textup{;}
\item \label{dd-h}
$u$ is \p-harmonic in $\Om \setm \{x_0\}$\textup{;}
\item \label{dd-sup}
$u(x_0)=\sup_\Om u$\textup{;}
\item \label{dd-inf}
$\inf_\Om u = 0$\textup{;}
\item \label{dd-Np0}
$\ut \in \Nploc(X \setm \{x_0\})$, where
\[
   \ut = \begin{cases}
     u & \text{in } \Om, \\
     0 & \text{on } X \setm \Om.
     \end{cases}
\]
\end{enumerate}
\medskip

A \emph{Green function} 
is a singular function 
which satisfies 
\begin{equation} \label{eq-normalized-Green-intro-deff}
\cp(\Om^b,\Om) = b^{1-p},
\quad \text{when }
   0  <b < u(x_0),
\end{equation}
where $\Om^b=\{x\in \Om:u(x)\ge b\}$.
\end{deff}

An earlier definition of singular functions along similar lines is due to 
Holo\-pai\-nen--Shan\-mu\-ga\-lin\-gam~\cite{HoSha}.
Under our assumptions and with natural interpretation for the
values at $x_0$ and in $X\setminus\Om$,
a Green function as in Definition~\ref{deff-sing} is
a singular function in the sense of~\cite{HoSha}, while
singular functions in~\cite{HoSha} are special cases of
our singular functions.
See \cite[Section~12]{BBLehGreen} for
a more precise comparison of these definitions.

The existence of singular and Green functions
in bounded domains
was studied in detail in~\cite{BBLehGreen}, and the following 
is one of the main results therein.
Note that the condition $\Cp(X\setm \Om)>0$ below 
is always true if $X$ is unbounded.
In fact, under the assumption $\Cp(X\setm \Om)>0$
conditions~\ref{dd-sup} and~\ref{dd-inf} are superfluous in Definition~\ref{deff-sing}
by~\cite[Theorem~1.6]{BBLehGreen}.

\begin{thm} \label{thm-main-intro}
\textup{(Theorem~1.3 in~\cite{BBLehGreen})}
Let $\Om \subset X$ be a bounded domain and let $x_0 \in \Om$.
Then there exists a Green function
in $\Om$ with singularity at $x_0$ if and only if $\Cp(X\setm \Om)>0$.
Moreover, if $u$ is a singular function in $\Om$ with singularity at $x_0$,
then there is a unique $\alp>0$ such that $\alp u$ is a Green function.
\end{thm}

\begin{remark} \label{rmk-gu}
Let $u$ be a singular function in a bounded domain $\Om$.
If $\Cp(\{x_0\})>0$, then $g_u\in L^p(\Om)$ by~\cite[Theorem~8.6]{BBLehGreen}.
Assume instead
that $\Cp(\{x_0\})=0$.
As $u$ is \p-harmonic in $\Om \setm \{x_0\}$
it belongs to $\Nploc(\Om \setm \{x_0\})$
and thus has a minimal \p-weak upper gradient $g_u \in L^p\loc(\Om \setm \{x_0\})$
in $\Om \setm \{x_0\}$.
Since $\Cp(\{x_0\})=0$,
Proposition~1.48 in \cite{BBbook} shows that
$g_u$ is also a \p-weak upper gradient of $u$ within $\Om$,
even though
$g_u\notin L^p\loc(\Om)$, because of
Theorem~\ref{thm-Green-cp-x0} below together with \ref{dd-Np0}.
As $\min\{u,k\}$ is a superminimizer in $\Om$
(and thus belongs to $\Np\loc(\Om)$), the function
\begin{equation}   \label{eq-def-Gu}
G_u:=\lim_{k \to \infty} g_{\min\{u,k\}}
\end{equation}
is a \p-weak upper gradient of $u$ which is minimal
in a certain sense, see  Kinnunen--Martio~\cite[Section~5]{KiMa} 
(or~\cite[Section~2.8]{BBbook}) for further details.
As
$u\in\Np\loc(\Om\setm\{x_0\})$, we have $G_u=g_u$ a.e.\
in $\Om\setm\{x_0\}$ and thus a.e.\ in $\Om$.
For singular functions $u$, we will therefore denote the minimal
\p-weak upper gradient by $g_u$ even within~$\Om$.
\end{remark}

\begin{thm} \label {thm-Green-cp-x0}
Let $u$ be a Green function in a bounded domain
$\Om$ with singularity at $x_0$.
Then the following are equivalent\/\textup{:}
\begin{enumerate}
\item \label{a-1}
$u(x_0)=\infty$\textup{;}
\item
$u$ is unbounded\textup{;}
\item
$u \notin \Np(\Om)$\textup{;}
\item \label{a-4}
$g_u \notin L^p(\Om)$\textup{;}
\item \label{a-5}
$\Cp(\{x_0\})=0$\textup{;}
\item \label{a-6} 
$\displaystyle\int_0^\de \biggl( \frac{\rho}{\mu(B_\rho)} \biggr)^{1/(p-1)} \,d\rho = 
\infty
\text{ for some \textup{(}or equivalently all\/\textup{)} $\de>0$.}$
\end{enumerate}
In particular, all of these statements
are true if $p< \uso$ and false if $p>\uso$.
\end{thm}

\begin{proof}
The statements~\ref{a-1}--\ref{a-5} were shown to
be equivalent in~\cite[Theorem~8.6]{BBLehGreen}. 
The equivalence~\ref{a-5} $\eqv$ \ref{a-6} follows from 
Proposition~\ref{prop-meta-zero-cap}\,\ref{prop-meta-zero-cap-a}
and the self-im\-prove\-ment of the \p-Poincar\'e inequality.
The last part follows from Proposition~\ref{prop-cp-x0}.
\end{proof}

\section{Pointwise estimates for Green functions}

\label{sect-ptwise-est}

\emph{Recall the general assumptions from
the beginning of Section~\ref{sect-harm}.}

\medskip

From Theorem~\ref{thm-metaestimate-int-only} 
and \cite[Theorem~1.5]{BBLehGreen}
we obtain the following pointwise estimates for Green functions.
Recall that by Theorem~\ref{thm-main-intro}, 
all estimates for Green functions in this section
also hold for singular functions $u$, but with 
comparison constants also depending on $u$.
For global Green functions on Riemannian manifolds,
estimates similar to~\eqref{eq-uappr-2}
were under various assumptions obtained in Holo\-pai\-nen~\cite[Section~5]{HoDuke}.

\begin{thm}  \label{thm-comp-u-R1-R2}
Let $0<R_1\le R_2 < \frac14 \diam X$ be fixed.
Assume that $u$ is a Green function with singularity at $x_0$ in
a domain $\Om$ such that 
\[
B_{R_1}\subset\Om\subset B_{R_2}.
\]
If  $r:=d(x,x_0) < R_1/50\la$, where $\la$ is the dilation constant 
in the Poincar\'e inequality, then
\begin{equation}
u(x)  
\simeq \cp(B_r,\Om)^{1/(1-p)} 
\simge \biggl(\frac{r^p}{\mu(B_{r})} \biggr)^{1/(p-1)} 
\label{eq-uappr-0}
\end{equation}
and
\begin{align}
u(x)&\simeq \cp(B_r,B_{R_1})^{1/(1-p)}  
     \simeq \cp(B_r,B_{R_2})^{1/(1-p)}   \label{eq-uappr-1}\\
&    \simeq 
    \int_{r}^{R_1} \biggl( \frac{\rho}{\mu(B_\rho)} \biggr)^{1/(p-1)} \,d\rho
   \simeq 
   \int_{r}^{R_2} \biggl( \frac{\rho}{\mu(B_\rho)} \biggr)^{1/(p-1)} \,d\rho,
\label{eq-uappr-2} 
\end{align}
with comparison constants depending only on $p$, 
the doubling constant of $\mu$, the constants in the Poincar\'e inequality,
and in \eqref{eq-uappr-1}--\eqref{eq-uappr-2} also on 
the quotient $R_2/R_1$.
\end{thm}

From now on we let $S_r=\{x:d(x,x_0)=r\}$.

\begin{proof}
The first comparison in~\eqref{eq-uappr-0} follows 
from \cite[Theorem~1.5]{BBLehGreen} when $x\in \bdy B_r$. 
To extend it to all $x \in S_r$,
we use 
\cite[Proposition~4.4]{BBLehGreen} with $r < \rho<\min\{2r,R_1/50\la\}$
and $K=\itoverline{B}_\rho\setm \tfrac12B_\rho$ as follows:
\[
\max_{\bdy B_r}u \le \max_{S_r}u \le \max_K u \le 
A\min_K u \le A \min_{S_r}u \le A\min_{\bdy B_r}u \le A\max_{\bdy B_r}u,
\]
where $A$ depends only on $p$, 
the doubling constant of $\mu$ and the constants in the Poincar\'e inequality.
The inequality in~\eqref{eq-uappr-0} then follows 
from~\cite[Proposition~5.1]{BBLeh1}.
Next, Lemma~11.22 in \cite{BBbook} yields
\[
\cp(B_r,B_{R_2})
\le \cp(B_r,\Om)
\le \cp(B_r,B_{R_1})
\simle \cp(B_r,B_{R_2}),
\]
and hence \eqref{eq-uappr-1} holds,
while \eqref{eq-uappr-2} follows directly
from Theorem~\ref{thm-metaestimate-int-only}.
\end{proof}

If $\mu$ is Ahlfors $Q$-regular around $x_0$
as in Theorem~\ref{thm-intro-Q-reg}, then, by \eqref{eq-uappr-2},
$u(x) \simeq r^{(p-Q)/(p-1)}$ if $p<Q$
and $u(x) \simeq \log (R_1/r)$ if $p=Q$,
while $u(x)$ is bounded if $p>Q$.
The estimates~\eqref{eq-uappr-1} 
can also be combined with the capacity estimates in~\cite{BBLeh1}  
to describe the pointwise behaviour of Green functions
as follows.

\begin{cor} \label{cor-BBL-conseq}
Assume that the assumptions in Theorem~\ref{thm-comp-u-R1-R2}
are satisfied, and in particular that $r:=d(x,x_0) < R_1/50\la$.
Then the following are true, with comparison constants 
independent of $u$, $x$, $r$ and $\Om$\textup{:}
\begin{enumerate}
\item 
If $p < \lqo$, then 
\begin{equation}\label{eq-compare-for-lQ}
 u(x) \simeq \biggl(\frac{r^p}{\mu(B_{r})}\biggr)^{1/(p-1)}.
\end{equation}
\item \label{kk-2}
If $p=\lqo \in \lQo$, then
\begin{align}               \label{eq-compare-for-maxlQ}
&\max\biggl\{\biggl(\frac{R_1^p}{\mu(B_{R_1})}\biggr)^{1/(p-1)}\log\frac{R_1}{r},
     \biggl(\frac{r^p}{\mu(B_{r})}\biggr)^{1/(p-1)}\biggr\} \nonumber \\
& \kern 10em \simle u(x) 
\simle \biggl(\frac{r^p}{\mu(B_{r})}\biggr)^{1/(p-1)}\log\frac{R_1}{r}.
\end{align}
\item 
If $p>q \in \lQo$, then 
\[ 
\biggl(\frac{r^p}{\mu(B_{r})}\biggr)^{1/(p-1)}
\simle 
 u(x) \simle \biggl(\frac{r^q}{\mu(B_{r})}\biggr)^{1/(p-1)}.
\] 
\item \label{kk-4}
If $p < s\in\uSo$, then 
\[
1 \simle u(x) \simle r^{(p-s)/(p-1)}.
\] 
\end{enumerate}
\end{cor}

\begin{proof}
This follows directly from Theorem~\ref{thm-comp-u-R1-R2}
together with 
\cite[Theorems~1.1 and~1.2 and  Propositions~6.2, 8.1 and~8.3]{BBLeh1}.
\end{proof}

Unweighted $\R^n$ with $p=n$ shows
sharpness of both the lower and upper bounds in \ref{kk-2},
and with $p<n=s$ of the upper bound in \ref{kk-4}.

Earlier, using the definition of singular functions in
Holo\-pai\-nen--Shan\-mu\-ga\-lin\-gam~\cite{HoSha},
and especially the a priori superlevel set property therein,
Danielli--Garofalo--Marola~\cite[Theorem~5.2]{DaGaMa} 
established \eqref{eq-compare-for-lQ}
for such functions.
In the borderline case $p=\lqo \in \lQo$,
they also gave an estimate which
is essentially equivalent to~\eqref{eq-compare-for-maxlQ}, since
the constant in~\cite[Theorem~5.2]{DaGaMa} in this case 
depends on $r^p/\mu(B_{r})$; cf.\
\cite[Theorem~3.1]{DaGaMa}.

Another consequence of~\eqref{eq-uappr-2} in 
Theorem~\ref{thm-comp-u-R1-R2} is that
\emph{all} Green functions with respect to
comparable open sets and comparable measures are comparable near the 
singularity in the following sense.

\begin{thm}  \label{thm-comp-u1-u2}
Let $0<R_1\le R_2 < \frac14 \diam X$ be fixed.
Assume that $\Om_1$ and $\Om_2$ are domains
such that 
\[
B_{R_1}\subset\Om_j\subset B_{R_2}, \quad j=1,2.
\]
Let $\mu_1$ and $\mu_2$ be doubling measures supporting 
\p-Poincar\'e inequalities on $X$.
Assume in addition that 
\begin{equation} \label{eq-comp-mu1-mu2}
\mu_1(B_\rho)\simle\mu_2(B_\rho) 
\quad \text{for all\/ } 0<\rho\le R_2.
\end{equation}
Also let $u_j$ be a Green function, with respect to $\mu_j$, in $\Om_j$ 
with singularity at $x_0$, $j=1,2$.
Then 
\begin{equation} \label{eq-comp-u1-u2}
u_1(x)\simge u_2(x) 
\quad \text{for all } x\in B_{R_1/50\la},
\end{equation}
with comparison constant depending only on $p$, $R_1$, $R_2$,
the doubling and Poincar\'e constants, and the
comparison constant in \eqref{eq-comp-mu1-mu2}.

Moreover, if $K\subset \Om_1\cap\Om_2$ is compact then 
\[
u_1(x)\simge u_2(x) 
\quad \text{for all } x\in K,
\]
with comparison constant also depending on $K$.
\end{thm}

\begin{proof}
By Theorem~\ref{thm-comp-u-R1-R2},
\[
      u_j(x) 
   \simeq 
    \int_{d(x,x_0)}^{R_1} \biggl( \frac{\rho}{\mu_j(B_\rho)} \biggr)^{1/(p-1)} \,d\rho,
    \ j=1,2,
\quad \text{for all } x\in B_{R_1/50\la}.
\]
Thus \eqref{eq-comp-u1-u2} follows from \eqref{eq-comp-mu1-mu2}.

For the last part, let 
$\Om=\Om_1 \cap \Om_2 \supset B_{R_1}$ and 
$0 <R < R_1/50\la$.
Since we have already shown \eqref{eq-comp-u1-u2} we may 
replace $K$ by $(K\setm B_R)\cup \bdy B_R$ and assume that
$\bdy B_R \subset K \subset \Om \setm B_{R}$.
As $\Om_1$ is connected, every component 
of $\Om_1 \setm\{x_0\}$ that intersects $K$
must contain a point in $\bdy B_{R}$, and since  $\bdy B_R$ 
is  compact there are only finitely many such components
$G_1,\ldots,G_m$ of $\Om_1 \setm \{x_0\}$.
By Harnack's inequality, as in Corollary~8.19 in~\cite{BBbook}, there
are constants $C_j$ such that
\[
     \sup_{K \cap G_j} v \le C_j \inf_{K \cap G_j} v
     \quad \text{for every positive \p-harmonic function $v$ in } G_j.
\]
Together with \cite[Proposition~4.4]{BBLehGreen},
this shows that $\sup_K u_1 \simle \inf_K u_1$.
Similarly $\sup_K u_2 \simle \inf_K u_2$,
which together with \eqref{eq-comp-u1-u2}
shows that $u_1 \simge u_2$ on $K$.
\end{proof}

Theorem~\ref{thm-comp-u1-u2} can be seen as a generalization of the well-known comparisons
for the classical Green functions in Euclidean domains for various linear
and nonlinear elliptic equations, cf.\ 
Littman--Stampacchia--Weinberger~\cite[Theorem~7.1]{LiStWe} and 
Serrin~\cite[Theorem~12]{Ser}.
Those estimates state that near the singularity, Green functions 
have certain predetermined growth and are thus comparable
to the fundamental solution for the Laplace or \p-Laplace equation,
see also Section~\ref{sect-elliptic-eq} below.
Such estimates have proved to be of great importance for both the 
interior and boundary regularity of such equations.
Theorem~\ref{thm-comp-u1-u2} and the other results in this section
provide us with similar comparisons
for Green functions associated
with the energy functionals $\int|\grad u|^p\,d\mu$ for large classes of
comparable measures.

On the contrary, the so-called quasiminimizers, 
introduced by
Giaquinta and Giusti~\cite{GG1},~\cite{GG2} as a natural unification
of differential equations with various ellipticities,
can (even in unweighted $\R^n$) have singularities
of arbitrary order, depending on the
quasiminimizing constant, see Bj\"orn--Bj\"orn~\cite{BBpower}.
In particular, quasiminimizers are not always solutions to partial
differential equations of \p-Laplacian type, since all such
solutions have comparable behaviour near their singularity, by
Serrin~\cite[Theorem~12]{Ser},
and thus also the same integrability.

\begin{remark} \label{rmk-assumptions}
Assume that $\Om$ is a bounded domain and
$B_{R}\subset\Om$.
Let $u$ be a Green function in $\Om$
with singularity at $x_0$.
If $\Cp(\{x_0\})>0$, then
$u$ is bounded
by Theorem~\ref{thm-Green-cp-x0}.
On the other hand, if $\Cp(\{x_0\})=0$ then
\eqref{eq-uappr-0}
implies that for all $0<r < R/50\la$ and
$x\in S_r$, 
\[ 
C_1 \cp(B_r,\Om)^{1/(1-p)} \le u(x) \le C_2 \cp(B_r,\Om)^{1/(1-p)},
\] 
where $C_1, C_2>0$ depend only on $p$, 
the doubling constant of $\mu$ and the constants in the Poincar\'e inequality,
but not on $u$, $x$, $r$ or $\Om$.
In particular, letting $r\nearrow R_2:=R/50\la$, 
we see that
\[
k:=\max_{\bdy B_{R_2}} u \le C_2 \cp(B_{R_2},\Om)^{1/(1-p)}
\]
and $u-k$ is a Green function in 
\[
\Om_k=\{x\in\Om: u(x)>k\}\subset B_{R_2}.
\]
As $\cp(\{x_0\},\Om)=0$ and $\cp$ is an outer capacity (by 
\cite[Theorem~6.19\,(vii)]{BBbook}),
we can now find $R_1>0$ such that
\[
\cp(B_{R_1},\Om)^{1/(1-p)} > \frac{C_2}{C_1} \cp(B_{R_2},\Om)^{1/(1-p)} \ge \frac{k}{C_1}.
\]
It follows that $\min_{\bdy B_{R_1}} u \ge k$ and hence 
$B_{R_1}\subset \Om_k \subset B_{R_2}$.
Note that $R_1$ depends only on $R_2$, $C_1$, $C_2$ and $\Om$, but not on $u$.
Since $B_{50\la R_2}=B_{R}\subset\Om\neq X$ by Theorem~\ref{thm-main-intro},
we also have that 
$50\la R_2 \le \diam X$
and hence $R_2<\tfrac14\diam X$.

Applying Theorem~\ref{thm-comp-u-R1-R2} to $u-k$, $\Om_k\subset B_{R_2}$
and $r<R_1/50\la$,
we thus see that all the estimates in this section 
hold near the singularity for Green functions in arbitrary
bounded domains $\Om$, even if $\Om$ is not contained in some
$B_{R_2}$ with $R_2<\tfrac14\diam X$.
Note that $\Cp(X \setm \Om)>0$, by Theorem~\ref{thm-main-intro},
since the Green function is presumed to exist.
\end{remark}

\section{Integrability of superharmonic functions}
\label{sect-int-superh}

\emph{Recall the general assumptions from
the beginning of Section~\ref{sect-harm}.}

\medskip

Green functions are particular examples of superharmonic functions.
Before studying special integrability properties of Green functions
and their minimal \p-weak upper gradients we
therefore recall general integrability results for superharmonic functions,
due to Kinnunen--Martio~\cite[Theorems~5.1 and~5.6]{KiMa}.
(On weighted $\R^n$, with a \p-admissible weight,
these results were earlier obtained by 
Heinonen--Kilpel\"ainen--Martio~\cite[Theorem~7.46]{HeKiMa}.)

We will then deduce a general result on integrability of the
minimal \p-weak upper gradients of superharmonic functions,
which will play a crucial role in Section~\ref{sect-int-gu}.
If $u$ is superharmonic, we define $G_u$ as in \eqref{eq-def-Gu}. 
That a superharmonic function fails to belong to $\Nploc(\Om)$
only if it is too large is a consequence of 
Proposition~7.4 in Bj\"orn--Bj\"orn--Parviainen~\cite{BBP}
(or~\cite[Corollary~9.6]{BBbook}).
Recall from Section~\ref{sect-exponents} that
\[
\utho= \inf \biggl\{\uth>0 :\, 
  \frac{\mu(B(x,r))}{\mu(B(x,R))} \simge \Bigl(\frac{r}{R}\Bigr)^{\uth} \\
\text{for all } x\in X \text{ and } 0<r\le R < 2 \diam X
        \biggr\}
\]
and that $\utho\ge1$, by Proposition~\ref{prop-uth-1}.

\begin{thm} \label{thm-superh-u-int+grad}
\textup{(\cite[Theorems~5.1 and 5.6]{KiMa})} 
Let $u$ be a superharmonic function in~$\Om$.
\begin{enumerate}
\item  \label{kj-ab}
If $p\le\utho$, then
$u\in L^\tau\loc(\Omega)$ 
and $G_u\in L^{t}\loc(\Omega)$
whenever
\[
0<\tau< \begin{cases}
  \dfrac{\utho(p-1)}{\utho-p}, & \text{if } p < \utho, \\
   \infty,  &\text{if } p =\utho,
\end{cases}
\qquad \text{and} \qquad
0< t <    \dfrac{\utho(p-1)}{\utho-1},
\]
respectively.
\item \label{kj-cg}
If $p>\utho$, then $u$ is continuous and thus locally bounded. 
Moreover, $G_u \in L^p\loc(\Omega)$ and $G_u=g_u$.
\end{enumerate}
In particular, it is always true that $u,G_u \in L^{p-1}\loc(\Om)$.
\end{thm}

\begin{proof}
In case~\ref{kj-ab}, it follows from Theorem~5.1 in
Haj\l asz--Koskela~\cite{HaKo} (or \cite[Theorem~4.21]{BBbook})
that so-called $(q,p)$-Poincar\'e
inequalities hold for every $1 \le q<\utho p/(\utho-p)$ 
(every $1 \le q<\infty$ if $p=\utho$).
Thus this part
follows directly from 
Theorems~5.1 and~5.6 in~\cite{KiMa}
(or \cite[Theorems~9.53 and 9.54]{BBbook}), upon letting $\uth \to \utho$.

In case \ref{kj-cg},
$u_k:=\min\{u,k\}$ is a superminimizer and thus belongs to $\Nploc(\Om)$.
It then follows from 
Corollary~5.39 in~\cite{BBbook}
that $u_k$ is continuous and that all points have positive capacity. 
In particular, considering $u_k$ with 
\[
\liminf_{x\to x_0} u(x) < k < \limsup_{x\to x_0} u(x)
\]
leads to a contradiction and hence $u$ is 
an $(-\infty,\infty]$-valued  continuous function.

On the other hand, by Proposition~2.2 in
Kinnunen--Shanmugalingam~\cite{KiSh06} (or Corollary~9.51 in \cite{BBbook}),
the set  $\{x\in\Om:u(x)=\infty\}$ has zero \p-capacity, and thus must be empty.
Hence $u$ is real-valued continuous, and therefore locally bounded.
Corollary~7.8 in Kinnunen--Martio~\cite{KiMa02} 
(or \cite[Corollary~9.6]{BBbook}) then shows that $u\in \Nploc(\Om)$, 
and in particular, $g_u=G_u \in L^p\loc(\Om)$.
\end{proof}

Using the ideas from the proof of Kinnunen--Martio~\cite[Theorem~5.6]{KiMa}, we obtain the following 
generalization of Theorem~\ref{thm-superh-u-int+grad}, 
which will be important when studying Green functions.

\begin{thm} \label{thm-int-Gu-superh-new}
Let $u$ be a superharmonic function in $\Om$.
\begin{enumerate}
\item  \label{kmm-a}
If $u \in L^\tau\loc(\Om)$ for all $0<\tau < \tau_0$, then 
$G_u\in L^t\loc(\Omega)$ whenever 
\[
  0<t< \begin{cases}
  \dfrac{p\tau_0}{\tau_0+1}, & \text{if } \tau_0<\infty, \\
  p, & \text{if } \tau_0=\infty.
  \end{cases}
\]
\item \label{kmm-b}
If $u$ is locally bounded, then
$G_u \in L^p\loc(\Omega)$.
\end{enumerate}
\end{thm}

It follows from Example~\ref{ex-log-2} below that even 
if $u \in L^{\tau_0}\loc(\Om)$ then it can happen that 
$G_u\notin L^{p\tau_0/(\tau_0+1)}\loc(\Omega)$,
which in particular shows that \ref{kmm-a} is sharp if $\tau_0 < \infty$.
That 
it is sharp also
when $\tau_0=\infty$
follows 
from unweighted $\R^n$ with $p=n$. 

\begin{proof}
\ref{kmm-b}
If $u$ is locally bounded then $u \in \Nploc(\Om)$,
by Corollary~7.8 in Kinnunen--Martio~\cite{KiMa02} (or \cite[Corollary~9.6]{BBbook}).
In particular, $g_u=G_u \in L^p\loc(\Omega)$.

\ref{kmm-a}
It suffices to consider $\tau_0<\infty$, since the infinite case is obtained by letting $\tau_0\to\infty$.
Let $B \subset 2B \Subset \Om$ be a ball and $0 < \eps < \tau_0(p-t)/t -1$.
Let 
\begin{equation}   \label{eq-def-trunc}
m=\min_{\overline{2B}} u > -\infty \quad \text{and} \quad
u_k=\min\{u-m,k\}+1, \quad k=1,2,\ldots,
\end{equation}
 which is a positive superminimizer in $2B$.
Then, using the Caccioppoli inequality for superminimizers
(Lemma~3.1 in Kinnunen--Martio~\cite{KiMa} or  \cite[Proposition~8.8]{BBbook}) 
with a suitable cut-off function $\eta\in\Lipc(2B)$, we obtain
\begin{align*}
   \int_B G_u^t \, d\mu
     &   = \lim_{k \to \infty}   
            \int_B g_{u_k}^t u_k^{-(1+\eps)t/p} u_k^{(1+\eps)t/p}\, d\mu \\
     & \le \lim_{k \to \infty}
           \biggl(    \int_B g_{u_k}^p u_k^{-(1+\eps)}\, d\mu \biggr)^{t/p}
          \biggl(\int_B  u_k^{(1+\eps)t/(p-t)}\, d\mu \biggr)^{1-t/p}  \\
     & \le C\biggl(    \int_{2B} (u-m+1)^{p-(1+\eps)}\, d\mu \biggr)^{t/p}
          \biggl(\int_B  (u-m+1)^{(1+\eps)t/(p-t)}\, d\mu \biggr)^{1-t/p} \\
     & < \infty, 
\end{align*}
where the first integral on the right-hand side
is finite by the last part of Theorem~\ref{thm-superh-u-int+grad},
while the last integral is finite by assumption.
\end{proof}

\section{Integrability of Green functions}
\label{sect-int-Green}

\emph{In addition to the general assumptions from
the beginning of Section~\ref{sect-harm},
we assume in this section that $\Om\subset X$ is a bounded domain.}

\medskip

In this section we study $L^\tau$-integrability 
(and nonintegrability) of Green functions.
In this case, we can improve upon the general
integrability results for superharmonic functions 
in Theorem~\ref{thm-superh-u-int+grad}.
We state the results here and  in the next
Section~\ref{sect-int-gu} for Green functions, but by Theorem~\ref{thm-main-intro}
the same conclusions hold for singular functions as well.
See also Section~\ref{sect-poles} for similar results for general
\p-harmonic functions with poles.

Note that, due to Theorem~\ref{thm-comp-u1-u2} and Remark~\ref{rmk-assumptions},
all Green functions  with the same
singularity $x_0$ belong to the same $L^\tau$ spaces.
If $p> \uso$, then
by Proposition~\ref{prop-cp-x0}\,\ref{e-c} and
Theorem~\ref{thm-Green-cp-x0} every Green function is bounded.
We therefore omit
this case in the rest of this section.
For $p\le\uso$ we have
the \emph{critical exponent}
\begin{equation}    \label{eq-def-tau-p}
\tau_p= \begin{cases}
     \displaystyle \frac{\uso(p-1)}{\uso-p}, & \text{if } p < \uso, \\
     \infty, & \text{if } p = \uso.
     \end{cases}
\end{equation}

\begin{thm} \label{thm-integrability-S}
Let $u$ be a Green function in $\Om$ with singularity at $x_0$.
If $p\le\uso$ then $u \in L^\tau(\Om)$ for all $0<\tau<\tau_p$.
\end{thm}

Since $\uso\le\utho$, we have for all $p\le\uso$ that 
$\tau_p\ge \utho(p-1)/(\utho-p)$, where the 
right-hand side is the borderline exponent for $\tau$ 
in Theorem~\ref{thm-superh-u-int+grad}
and the inequality is strict when $\uso<\utho$.
Hence Green functions have higher integrability than
what is known for general
superharmonic functions when
$\uso<\utho$.

\begin{remark} \label{rmk-DMG}
Danielli--Garofalo--Marola~\cite[Corollary~5.4]{DaGaMa} showed 
for singular functions $u$, as defined in
Holo\-pai\-nen--Shan\-mu\-ga\-lin\-gam~\cite{HoSha}, that
$u \in L^\tau(\Om)$ whenever
\[
p < \lqq:= \sup \lQ 
\quad\text{and}\quad
\tau < \frac{\lqq(p-1)}{\uthtilde-p},
\quad\text{where }
\uthtilde=\log_2 C_\mu \in \uTh,
\]
$C_\mu$ is the doubling constant of $\mu$ and 
\begin{equation} \label{eq-qt}
  \lQ=\biggl\{q>0 : \text{there is $C$ so that } 
        \frac{\mu(B_r)}{\mu(B_R)}  \le C \Bigl(\frac{r}{R}\Bigr)^q 
        \text{ for } 0 < r < R <\infty
        \biggr\}.
\end{equation}
Note that $\lqq \le \lqo$, where the inequality can be strict
as the range in \eqref{eq-qt} is $0<r<R<\infty$.
They however also implicitly assumed that $\lqq \in \lQ$,
see \cite[eq.\ (2.2)]{DaGaMa}, and that $X$ is 
linearly locally connected (LLC),
through their use (at the bottom of p.~354) of Lemma~5.3 in 
Bj\"orn--MacManus--Shanmugalingam~\cite{BMS}.
\end{remark}

Note that $\lqq$ in~\cite{DaGaMa} can be much smaller than our
$\uso$, which in turn can be much smaller than $\uthtilde$.
Thus, both the numerator and the denominator are in general worse
in~\cite{DaGaMa} than in the critical exponent~\eqref{eq-def-tau-p},
and the range of possible exponents $p$
is smaller than here. 
Thus Theorem~\ref{thm-integrability-S} 
is a substantial improvement upon the results in~\cite{DaGaMa}. 
Moreover, Theorem~\ref{thm-integrability-S} 
is sharp (up to certain borderline cases), 
by Theorem~\ref{thm-nonintegrability-S}.

\begin{proof}[Proof of Theorem~\ref{thm-integrability-S}]
We may assume that $B_{R_1} \subset \Om \subset B_{R_2}$, 
where  $0<R_1\le R_2 < \frac14 \diam X$, see Remark~\ref{rmk-assumptions}.
Let $r_k=2^{-k}R_1/50\la$ and $B^k=B_{r_k}$, $k=1,2,\ldots$\,.
By the assumptions on $\tau$, 
we find $\us > \uso$ such that $\tau<\us(p-1)/(\us-p)=:\be$.
In particular, $\us \in \uSo$.

Consider $x$ such that $r_{k+1} \le r:=d(x,x_0) < r_k$.
Then, by Theorem~\ref{thm-comp-u-R1-R2} and since $\us \in \uSo$,
\begin{align*}
    u(x) 
   & \simeq 
\int_{r}^{R_1} \biggl( \frac{\rho^p}{\mu(B_\rho)^{1-p/\us} \mu(B_\rho)^{p/\us}}
              \biggr)^{1/(p-1)} \,\frac{d\rho}{\rho} \\
& \simle \biggl( \frac{1}{\mu(B_r)^{1-p/\us}} \biggr)^{1/(p-1)}
      \int_{r}^{R_1} \frac{d\rho}{\rho} = \frac{\log (R_1/r)}{\mu(B_r)^{1/\be}}
\simle\frac{k}{\mu(B^{k})^{1/\be}}.
\end{align*}
As $\tau /\be <1$, we have for
any $\ls \in \lSo$ that 
\[
\int_{B^{1}} u^\tau \, d\mu
\simle \sum_{k=1}^\infty k^\tau \mu(B^k)^{1-\tau /\be}
\simle \sum_{k=1}^\infty k^\tau 2^{-k\ls (1-\tau /\be)}
<\infty,
\]
where we recall that $\lSo\neq\emptyset$ under our assumptions.
Hence $u \in L^\tau(\Om)$.
\end{proof}

Next we consider nonintegrability of $u$.

\begin{thm} \label{thm-nonintegrability-S}
Let $u$ be a Green function in $\Om$ with singularity at $x_0$
and assume that $p \le \uso$.
\begin{enumerate}
\item \label{hh-a}
If $\tau > \tau_p$ then $u \notin L^\tau(\Om)$.
\item \label{hh-b}
If $\uso \notin \uSo \setm \lSo$,
then $u \notin L^{\tau_p}(\Om)$.
\end{enumerate}
\end{thm}

\begin{proof}
By Remark~\ref{rmk-assumptions},
we may assume that $B_{R_1} \subset \Om \subset B_{R_2}$, 
where $0<R_1\le R_2 < \frac14 \diam X$.
Let $r_k=2^{-k}R_1/50\la$ and $B^k=B_{r_k}$, $k=0,1,\ldots$\,.

If $p=\uso$, then $\tau_p=\infty$ and there is nothing to
prove in~\ref{hh-a}, 
while Proposition~\ref{prop-cp-x0} and Theorem~\ref{thm-Green-cp-x0}
show that $u \notin L^\infty(\Om)$ in~\ref{hh-b}.

Assume now that $p<\uso$.
In case~\ref{hh-a}, let $s>p$ be such that $\tau\ge s(p-1)/(s-p)$
and $s\notin \uSo$.
In case~\ref{hh-b}, we instead take $s=\uso$ and $\tau=\tau_p$.
In both cases, there is a sequence $k_j \nearrow \infty$
such that $\mu(B^{k_j}) \simle r_{k_j}^s$.
As $\tau \ge \tau_p > p-1$,
we therefore obtain from \eqref{eq-uappr-0} and 
the reverse-doubling property of $\mu$ that
\[
\int_{B^0} u^\tau \, d\mu 
   \simge \sum_{k=0}^\infty \biggl( \frac{r_k^p}{\mu(B^{k})} 
            \biggr)^{\tau/(p-1)}  \mu(B^{k}) 
   \simge \sum_{j=0}^\infty r_{k_j}^{p\tau/(p-1)+s(1-\tau/(p-1))}  = \infty,
\]
since the last exponent is nonpositive.
\end{proof}

Thus, when $p < \uso$
we know exactly which $L^\tau$-integrability $u$ has, 
apart from the borderline case $\tau=\tau_p$.
When $p=\uso$
we also lack a complete characterization of when $u$ is bounded.
This is however natural, as
knowing $\lSo$ and $\uSo$
is not enough to determine integrability and boundedness in
these cases,
see the last part with $p=s$ in
Example~\ref{ex-log} below
and the comment 
before Lemma~\ref{lem-suff-cap-zero}.
At the same time, under the assumption 
$\uso \notin \uSo \setm \lSo$
we have a complete characterization, as follows.

\begin{cor} \label{cor-integrability-S-extra}
Let $u$ be a Green function in $\Om$ with singularity at $x_0$ 
and assume that
$p \le \uso \notin \uSo \setm \lSo$.
Then $u$ is unbounded, and $u \in L^\tau(\Om)$ 
if and only if $\tau < \tau_p$.
\end{cor}

\begin{proof}
By Proposition~\ref{prop-cp-x0},
$\Cp(\{x_0\})=0$
and hence $u$ is unbounded
by Theorem~\ref{thm-Green-cp-x0}.
The rest of the conclusion follows directly from
Theorems~\ref{thm-integrability-S} and~\ref{thm-nonintegrability-S}.
\end{proof}

The following more general characterizations can be used also when 
the critical case is not captured by the $S$-sets,
and hence they complement Corollary~\ref{cor-integrability-S-extra}
when $\uso \in \uSo \setm \lSo$.

\begin{thm} \label{thm-u-Q}
Let $u$ be a Green function in $\Om$ with singularity at $x_0$.
\begin{enumerate}
\item \label{a-b}
If  $p <\uso$, then $u \in  L^{\tau_p}(\Om)$ if and only if
\begin{equation} \label{eq-prop-u-Q}
  \sum_{k=1}^\infty \biggl(\frac{2^{-k\uso}}{\mu(B_{2^{-k}})}\biggr)^{p/(\uso-p)}
  < \infty,
\end{equation}
or equivalently
\[
   \int_0^1 \biggl(\frac{\rho^{\uso}}{\mu(B_{\rho})}\biggr)^{p/(\uso-p)} \frac{d\rho}{\rho}
   <\infty.
\]
\item \label{b-b}
If $p=\uso$, then $u$ is bounded 
{\rm(}i.e.\ $u\in L^{\tau_p}(\Om)$\/{\rm)} if and only if
\begin{equation} \label{eq-prop-u-1}
 \sum_{k=1}^\infty \biggl( \frac{2^{-k\uso}}{\mu(B_{2^{-k}})} 
 \biggr)^{1/(\uso-1)}  < \infty,
\end{equation}
or equivalently 
\[
\int_0^1 \biggl(\frac{\rho^{\uso}}{\mu(B_{\rho})}\biggr)^{1/(\uso-1)} 
\frac{d\rho}{\rho}
=    \int_0^1 \biggl(\frac{\rho}{\mu(B_{\rho})}\biggr)^{1/(\uso-1)} \, d\rho
   <\infty.
\]
\end{enumerate}
\end{thm}
\medskip

Perhaps surprisingly, condition~\eqref{eq-prop-u-1} is not the limit 
of~\eqref{eq-prop-u-Q}, 
as $p\to\uso$, but coincides with 
\eqref{eq-prop-u-Q} for $p=1$ 
(which however is not allowed in \ref{a-b}).

\begin{proof}
The equivalences between the sums and integrals are obvious.
Part \ref{b-b} follows from 
Proposition~\ref{prop-meta-zero-cap} and Theorem~\ref{thm-Green-cp-x0},
so we turn to \ref{a-b}.

By Remark~\ref{rmk-assumptions},
we may assume that $B_{R_1} \subset \Om \subset B_{R_2}$, 
where $0<R_1\le R_2 < \frac14 \diam X$.
Let $r_k=2^{-k}R_1/50\la$ and $B^k=B_{r_k}$, $k=0,1,\ldots$\,.
Consider $x$ such that $r_{k+1} \le r:=d(x,x_0) < r_k$.
By 
\eqref{eq-uappr-0},
\[
u(x) \simge  \biggl(\frac{r_k^{p}}{\mu(B^k)}\biggr)^{1/(p-1)}.
\]
Note from \eqref{eq-def-tau-p} that
\begin{equation} \label{eq-taup/p-1}
\frac{\tau_p}{p-1} = \frac{\uso}{\uso-p}
\quad \text{and} \quad
  \frac{\tau_p}{p-1}-1 = \frac{p}{\uso-p}.
\end{equation}
Hence, using that $\mu$ is reverse-doubling with $\xi=2$,
\begin{equation}\label{eq-lower-for-t_p-int}
\int_{B^1} u^{\tau_p} \, d\mu
\simge 
\sum_{k=1}^\infty \biggl(\frac{r_k^p}{\mu(B^{k})}\biggr)^{\tau_p/(p-1)} 
       \mu(B^{k}\setm B^{k+1})
\simeq \sum_{k=1}^\infty \biggl(\frac{r_k^{\uso}}{\mu(B^{k})}\biggr)^{p/(\uso-p)},
\end{equation}
which diverges if and only if the sum in \eqref{eq-prop-u-Q} diverges. 

Conversely, we apply~\eqref{eq-uappr-2} to $u(x)$ and obtain
\begin{equation}   \label{eq-2-sum-with-meta-first}
\int_{B^1} u^{\tau_p} \, d\mu
\simle \sum_{k=1}^\infty \biggl( \sum_{j=1}^k 
     \biggl(\frac{r_j^p}{\mu(B^{j})} \biggr)^{1/(p-1)} \biggr)^{\tau_p} 
\mu(B^k\setm B^{k+1}).
\end{equation}
We distinguish two cases. 
If $\tau_p\le1$, then \eqref{eq-2-sum-with-meta-first} gives
\begin{align*} 
\int_{B^1} u^{\tau_p} \, d\mu
&\simle \sum_{k=1}^\infty \sum_{j=1}^k 
     \biggl(\frac{r_j^p}{\mu(B^{j})} \biggr)^{\tau_p/(p-1)} \mu(B^k\setm B^{k+1})\\
&= \sum_{j=1}^\infty
     \biggl(\frac{r_j^p}{\mu(B^{j})} \biggr)^{\tau_p/(p-1)} 
     \sum_{k=j}^\infty \mu(B^k\setm B^{k+1}) 
= \sum_{j=1}^\infty
     \biggl(\frac{r_j^p}{\mu(B^{j})} \biggr)^{\tau_p/(p-1)} \mu(B^j),
\end{align*}
which, by \eqref{eq-taup/p-1}, is
the same as
the last sum 
in \eqref{eq-lower-for-t_p-int}.

If $\tau_p>1$, we rewrite \eqref{eq-2-sum-with-meta-first} as
\begin{equation}   \label{eq-2-sum-with-meta}
\int_{B^1} u^{\tau_p} \, d\mu
\simle \sum_{k=1}^\infty r_k^\eps \biggl( \sum_{j=1}^k 
     \frac{r_j^{p/(p-1)-\eps/\tau_p}}{\mu(B^{j})^{1/(p-1)-1/\tau_p}} 
          \biggl(\frac{r_j^\eps \mu(B^{k})}{r_k^\eps \mu(B^{j})} 
                              \biggr)^{1/\tau_p} \biggr)^{\tau_p},
\end{equation}
where  $\eps>0$.
Now, for any $0<q\in\lQo$,
we have
\[
\frac{r_j^\eps \mu(B^{k})}{r_k^\eps \mu(B^{j})} 
    \simle  \Bigl(\frac{r_k}{r_j}\Bigr)^{q-\eps}
    =2^{(j-k)(q-\eps)}.
\]
The inner sum on the right-hand side of \eqref{eq-2-sum-with-meta} 
is then estimated using \eqref{eq-taup/p-1} 
and H\"older's inequality with exponents $\tau_p$
and $\tau_p/(\tau_p-1)$, 
as follows
\[
\sum_{j=1}^k \simle \biggl( \sum_{j=1}^k 
        \frac{r_j^{\uso p/(\uso-p)-\eps}}{\mu(B^{j})^{p/(\uso-p)}} \biggr)^{1/\tau_p}
\biggl( \sum_{j=1}^k 2^{(j-k)(q-\eps)/(\tau_p-1)} \biggr)^{1-1/\tau_p}.
\]
For $q>\eps$, the last sum 
is bounded from above by a constant, independent of $k$.
Hence, by inserting the previous 
estimate into~\eqref{eq-2-sum-with-meta}
and changing the order of summation, we obtain 
\begin{align*}
\int_{B^1} u^{\tau_p} \, d\mu 
    &\simle \sum_{k=1}^\infty r_k^\eps \sum_{j=1}^k 
         \biggl(\frac{r_j^{\uso}}{\mu(B^{j})} \biggr)^{p/(\uso-p)}  r_j^{-\eps} \\
    &= \sum_{j=1}^\infty r_j^{-\eps} \biggl(\frac{r_j^{\uso}}{\mu(B^{j})} 
                \biggr)^{p/(\uso-p)}   \sum_{k=j}^\infty r_k^\eps 
   \simeq \sum_{j=1}^\infty \biggl(\frac{r_j^{\uso}}{\mu(B^{j})} \biggr)^{p/(\uso-p)}.
   \qedhere
\end{align*}
\end{proof}

Next we compare Green functions for different parameters $p$
under a natural assumption on the validity of Poincar\'e inequalities.

\begin{cor}  \label{cor-different-p}
Assume that $\mu$ supports a $p_0$-Poincar\'e inequality
for some $1\le p_0<\uso$.
For each $p_0<p\le \uso$, let $u_p$ be a Green function in 
$\Om$ with singularity at $x_0$ with
respect to $p$.
Then the following are true\/\textup{:}
\begin{enumerate}
\item \label{l-a}
 If $p_0<p_1 < p_2 < \uso$ and
 $u_{p_1} \in  L^{\tau_{p_1}}(\Om)$, then $u_{p_2} \in  L^{\tau_{p_2}}(\Om)$.
\item \label{l-b}
If $u_{\uso}$ is bounded, then $u_p\in L^{\tau_p}(\Om)$
for all $p_0<p<\uso$.
\end{enumerate}
\end{cor}

Note that 
$\tau_{p_1}< \tau_{p_2} <\infty$ in \ref{l-a}.
Example~\ref{ex-log} below shows that the converse implications
can fail,
and that it is possible to have $u_p \in L^{\tau_p}(\Om)$
for all $1 <p<\uso$ but not for $p=\uso$.

\begin{proof}
\ref{l-a}
Let $a_k=2^{-k\uso}/\mu(B_{2^{-k}})$.
Then by Theorem~\ref{thm-u-Q},
$u_{p_j} \in L^{\tau_{p_j}}(\Om)$ if and only if
\[
\sum_{k=1}^\infty a_k^{p_j/(\uso-p_j)}<\infty.
\]
Hence, 
if $u_{p_1} \in  L^{\tau_{p_1}}(\Om)$, then 
\[
\sum_{k=1}^\infty a_k^{p_2/(\uso-p_2)}
\le \biggl( \sum_{k=1}^\infty a_k^{p_1/(\uso-p_1)} \biggr)^{\be_2/\be_1}
< \infty,
\]
where $\be_2:=p_2/(\uso-p_2) > p_1/(\uso-p_1)=:\be_1$.
The proof of \ref{l-b} is similar upon noting that
condition~\eqref{eq-prop-u-1} is condition~\eqref{eq-prop-u-Q}
with $p=1$ in Theorem~\ref{thm-u-Q}.
\end{proof}

\begin{example} \label{ex-log}
Let $s>1$, $\be >0$ and $n \ge 2$.
Consider 
$\R^n$ equipped with the measure
$d\mu=w\,dx$, where
(by abuse of notation)
$w(x)=w(|x|)$ 
and
\[
      w(\rho)=\begin{cases}
          \rho^{s-n} |{\log \rho}|^\be, & \text{if } 0<\rho\le1/e, \\ 
          \rho^{s-n}, & \text{otherwise}.
        \end{cases}
\]
Let $u$ be a Green function with singularity at $x_0=0$
in a bounded domain $\Om \ni x_0$.
By \cite[Proposition~10.5 and Remark~10.6]{BBLeh1},
$\mu$ is doubling and supports a $1$-Poincar\'e inequality,
i.e.\ $w$
is $1$-admissible on $\R^n$.
Moreover, by Example~3.1 in \cite{BBLeh1},
\[
\lSo=\lQo=(0,s)
\quad \text{and} \quad
\uSo=\uQo=[s,\infty).
  \]
In particular, $\lso=\lqo=\uso=\uqo=s$ and $\uso\in \uSo\setm \lSo$,
so the assumption on $\uso$ in 
Theorem~\ref{thm-nonintegrability-S}\,\ref{hh-b} and
Corollary~\ref{cor-integrability-S-extra}
fails.

It was also observed in Example~3.1 in \cite{BBLeh1}
that
$\mu(B_r)\simeq r^s |{\log r}|^\be$ for $r \le 1/e$.
Thus 
\begin{equation}   \label{eq-est-mu-B-k}
\frac{2^{-ks}}{\mu(B_{2^{-k}})} \simeq \frac{1}{k^\be}.
\end{equation}
Hence, by Theorem~\ref{thm-u-Q}\,\ref{a-b}, for $p<s$,
\[
u \in L^{\tau_p}(\Om,w)
\quad \text{if and only if} \quad
\frac{\be p}{s-p}>1, \text{ i.e. }
\frac{s}{1+\be} < p < s,
\]
showing that the sets $\uSo$ and $\lSo$ themselves are not fine enough
to determine the borderline $L^{\tau_p}$-integrability of Green functions.
In particular, if $\be \ge s-1$, then
$u \in L^{\tau_p}(\Om,w)$ for all $1<p<s$.

For $p=s$, Theorem~\ref{thm-u-Q}\,\ref{b-b} and \eqref{eq-est-mu-B-k}
show that $u$ is bounded if and only if $\be > s-1$.
In particular, if $\be =s-1$, then $u \in L^{\tau_p}(\Om,w)$ for
all $1<p< s$, but $u \notin L^{\tau_p}(\Om,w)$ when $p=s$.
\end{example}

\section{Integrability of gradients of Green functions}
\label{sect-int-gu}

\emph{In addition to the general assumptions from
the beginning of Section~\ref{sect-harm},
we assume in this section that $\Om\subset X$
 is a bounded domain.}

\medskip
  
In this section we turn to 
the $L^t$-integrability of the minimal \p-weak upper
gradient $g_u$ 
of a Green function $u$.
See Remark~\ref{rmk-gu} for how to interpret $g_u$.
If $p>\uso$ then $g_u \in L^p(\Om)$,
by Proposition~\ref{prop-cp-x0}\,\ref{e-c} and
Theorem~\ref{thm-Green-cp-x0}, and we therefore omit
this case in the rest of this section. 
Thus, in particular, we assume that  $\uso>1$.
The exponent 
\[
t_p=\frac{\uso(p-1)}{\uso-1}
\]
will be critical.
Note that
\begin{equation} \label{eq-tau0-taup}
      t_p=\frac{p \tau_p}{\tau_p+1}< \tau_p
      \quad \text{if } p < \uso
\end{equation}
and that $t_p \ge 1$ if and only if $p \ge 2-1/\uso$.

\begin{thm}\label{thm-integrability-gu}
Let $u$ be a Green function in $\Omega$ with singularity at $x_0$
and assume that $p \le \uso$.
Then 
$g_u\in L^t(\Omega)$ 
whenever $0<t<t_p$.

Moreover, $u \in N^{1,t}(\Om)$ whenever $1 \le t<t_p$.
\end{thm}

Since $\uso\le\utho$, we have for all $p\le\uso$ that 
$t_p\ge \utho(p-1)/(\utho-1)$, where the 
right-hand side is the borderline exponent $t$ 
in Theorem~\ref{thm-superh-u-int+grad}
and the inequality is strict when $\uso<\utho$.
Hence the minimal \p-weak upper
gradients of Green functions have higher integrability than 
what is known for
the minimal \p-weak upper gradients of general
superharmonic functions when
$\uso<\utho$.

It follows from 
Danielli--Garofalo--Marola~\cite[Corollary~5.4]{DaGaMa}
that $g_u \in L^t(\Om)$ if
 $p <\lqq$ and $t <\lqq(p-1)/(\uthtilde-1)$,
where $\lqq$,
$\uthtilde$ 
and $X$ are as in 
Remark~\ref{rmk-DMG}. 
Thus Theorem~\ref{thm-integrability-gu} 
is a substantial improvement upon the results in~\cite{DaGaMa}. 
Moreover, Theorem~\ref{thm-integrability-gu} 
is sharp by Theorem~\ref{thm-nonintegrability-gu},
at least when $t_p \ge 1$ and $X$ supports a  
$t_p$-Poincar\'e inequality at $x_0$ for small radii.

\begin{proof}[Proof of Theorem~\ref{thm-integrability-gu}]
The first part follows directly from
Theorem~\ref{thm-integrability-S}
together with Theorem~\ref{thm-int-Gu-superh-new},
\eqref{eq-tau0-taup} and \ref{dd-Np0} in Definition~\ref{deff-sing}.
If $1 \le t < t_p$, then $g_u$ is also
a $t$-weak upper gradient of $u$ (although not necessarily minimal).
Moreover,  $u \in L^t(\Om)$ by Theorem~\ref{thm-integrability-S}
and thus $u \in N^{1,t}(\Om)$.
\end{proof}

For $p=\uso$ we get the following consequences,
when combining Theorem~\ref{thm-integrability-gu} with
Theorem~\ref{thm-Green-cp-x0}.

\begin{cor} \label{cor-nonint-gu-t=p}
Assume that $p=\uso$.
\begin{enumerate}
    \item \label{a-aa}
If $\Cp(\{x_0\})=0$, then $g_u \in L^t(\Om)$ if and only if $0<t<p$.
  \item \label{a-bb}
If $\Cp(\{x_0\})>0$, then $g_u \in L^p(\Om)$.
\end{enumerate}
\end{cor}

If $\mu$ supports a suitable global Poincar\'e inequality, then
nonintegrability results for $g_u$ can
be obtained by combining Theorem~\ref{thm-nonintegrability-S} and
the Sobolev inequality (see~\cite[Corollary~4.23]{BBbook})
determined by the global exponent set $\uTh$.
For instance, in the case when $p<Q$ and $\mu$ is globally Ahlfors $Q$-regular,
so that $\uso=Q\in\uTh$,
and $\mu$
supports a global $t$-Poincar\'e inequality for $t\ge \max\{1,t_p\}$,
this already leads to the 
nonintegrability results 
in Theorem~\ref{thm-intro-Q-reg}\,\ref{i-c},
since otherwise 
the Sobolev inequality 
together with $g_u\in L^t(\Om)$
would imply that $u$ is bounded
(when $t\ge Q$) or that 
\[
u\in L^{Qt/(Q-t)}(\Om)\subset L^{\tau_p}(\Om) \quad \text{(when $t_p\le t<Q$).}
\]
When $t=t_p$, this is sharp by 
Theorem~\ref{thm-intro-Q-reg}\,\ref{i-b}.
Using the Sobolev inequality from
\cite[Theorem~5.1]{BBsemilocal} makes
it possible to instead use a local $t$-Poincar\'e inequality
and a local exponent set 
(defined similarly 
to the global exponent set $\uTh$).
However, considering the $1$-admissible power weight $w(x)=|x|^{-\alp}$ 
on $\R^n$, as in Example~\ref{ex-power}, with $0 < \alp < n-1$ and $x_0=0$,
shows that such a direct application of Sobolev inequalities does not
lead to the results in 
Theorems~\ref{thm-intro}\,\ref{intro-f}
and \ref{thm-intro-Q-reg}\,\ref{i-c}, which use 
the pointwise exponent set $\uSo$
(and only assume a pointwise $t$-Poincar\'e inequality). Indeed,
in this case $\uso=n-\alp<n$ while the infimum
of the local exponent set is $n$.
(If $n-1 \le \alp <n$, then $\uso \le 1 <p$ and the Green
function $u$ is bounded and $g_u \in L^p(\Om)$.)

More generally, under our weaker assumptions, for
the nonintegrability of $g_u$ it is convenient 
to use one more exponent,
\[
\qhat=\lqo+1-\frac{\lqo}{\uso}.
\]
The reason for introducing $\qhat$ is that $t_p<\lqo$ if and only if
$p < \qhat$, which will be important below.
Note that $\qhat\le\uso$, with equality if and only if $\lqo=\uso$,
since $\uso >1$.

\begin{thm}\label{thm-nonintegrability-gu}
Let $u$ be a Green function in $\Omega$, with singularity at $x_0$.
\begin{enumerate}
\item \label{m-a}
If $p < \uso$, then $g_u\notin L^t(\Omega)$ 
whenever 
$\mu$ supports a $t$-Poincar\'e inequality at $x_0$
for small radii, $t>t_p$ and $t \ge 1$.
\item \label{m-c}
If $p< \qhat$, $t_p \ge 1$ and $\uso \notin \uSo \setm \lSo$, then
$g_u\notin L^{t_p}(\Omega)$ whenever 
$\mu$ supports a $t_p$-Poincar\'e inequality at $x_0$ 
for small radii.
\item \label{m-d}
If $\qhat \le p< \uso\notin \uSo \setm \lSo$ and 
$\lqo>1$,
then 
$g_u\notin L^{t_p}(\Omega)$ whenever 
$\mu$ supports a $t_0$-Poincar\'e inequality at $x_0$ for 
small radii and some $1\le t_0<t_p$.
\end{enumerate}
\end{thm}

Note that $\lqo > t_p \ge 1$ in~\ref{m-c},
while $\lqo >  1$  needs to be assumed explicitly in~\ref{m-d}.

\begin{proof}
By Remark~\ref{rmk-assumptions},
we may assume that $B_{R_1} \subset \Om \subset B_{R_2}$, 
where $0<R_1\le R_2 < \frac14 \diam X$.
The strong minimum principle for superharmonic functions
(Theorem~9.13 in \cite{BBbook}) and \eqref{eq-uappr-0}
imply that for $0<r < R_1/50\la$,
\begin{equation}   \label{eq-lower-for-m_r}
   m_r := \inf_{B_r} u = \min_{\bdy B_r} u
\simge \biggl(\frac {\mu(B_r)}{r^p} \biggr)^{1/(1-p)}.
\end{equation}

\ref{m-a}
Note that $t_p< p < \uso$.
Let $q>\max\{t,\uso\}$.
Since  $\min\{1,u/m_r\}$ is admissible for
$\capp_t(B_r,\Om)$, 
we obtain by \cite[Proposition~8.3]{BBLeh1} that
\begin{equation}   \label{eq-est-from-prop-8.3}
\int_{\Om\setm B_r} g_u^t \,d\mu 
\ge m_r^t \capp_t(B_r,\Om)
\ge m_r^t \capp_t(B_r,B_{R_2})
\simge \biggl(\frac {\mu(B_r)}{r^p} \biggr)^{t/(1-p)} r^{q-t}.
\end{equation}
If $s < \uso$, 
then there is a sequence 
$r_j \searrow 0$ such that $\mu(B_{r_j}) \simle r_j^s$. 
For this sequence, \eqref{eq-est-from-prop-8.3} becomes
\[
\int_{\Om\setm B_{r_j}} g_u^t \,d\mu \simge r_j^{(s-p)t/(1-p)+q-t} \to\infty,
\quad \text{as }j\to 0,
\]
provided that the exponent $(s-p)t/(1-p)+q-t<0$,
which is equivalent to
\[
q<\frac{t(s-1)}{p-1}.  
\]
Clearly, for every $t>t_p$, this is satisfied for some 
$q>\max\{t,\uso\}$ and $s < \uso$,
and \ref{m-a} follows.

\ref{m-c}
Since $\uso \in \lSo $ or $\uso \notin \uSo$,
there is a sequence $r_k \searrow 0$ such that
$\mu(B_{r_k}) \simle r_k^{\uso}$, $k=0,1,\ldots$\,.
Let $B^k=B_{r_k}$ and write $a_k = m_{r_k}$.
Because $\lim_{r\to0}m_r=\infty$, we can also assume that
\[
\tfrac12 a_{k+1} \ge A_k := \max_{\bdy B^k} u,
\]
$r_{k+1} < \frac12 r_k$ and that 
$r_0< R_1/50\la$ is so small that 
the $t_p$-Poincar\'e inequality at $x_0$ holds for radii up to
$2r_0$.
This will be important below when we use the 
capacity estimates from \cite{BBLeh1}.
Then for all $k=1,2,\ldots$\,,
\[
    v_k=\frac{(\min\{u,a_{k}\}-A_{k-1})_\limplus}{a_{k}-A_{k-1}}
\]
is admissible for $\ctp(B^{k},B^{k-1})$.
Moreover, $a_k-A_{k-1} \ge \tfrac12 a_k$.
As $g_u$ is also
a $t_p$-weak upper gradient of $u$ (although not necessarily minimal), 
we have
\begin{equation} \label{eq-ak}
  \int_{B^{k-1} \setm B^{k}} g_{u}^{t_p}\,d\mu
\ge
(a_{k}-A_{k-1})^{t_p} \int_{B^{k-1} \setm B^{k}} g_{v_k}^{t_p}\,d\mu
\simge     a_k^{t_p} \ctp(B^{k},B^{k-1}).
\end{equation}

Since $p < \qhat$, we see that $t_p < \lqo$.
As $\mu$ supports a $t_p$-Poincar\'e inequality at $x_0$
for radii up to $2r_0$,
we get by~\cite[Proposition~6.1]{BBLeh1},
using also 
\eqref{eq-lower-for-m_r} and $t_p \ge p-1$, that
\begin{align*}
\int_{B^1} g_u^{t_p} \, d\mu 
& \simge \sum_{k=1}^{\infty}  a_{k}^{t_p} \ctp(B^{k},B^{k-1})
  \simge \sum_{k=1}^{\infty}  \biggl(\frac {r_k^p}{\mu(B^k)}\biggr)^{t_p/(p-1)}
    \frac{\mu(B^{k})}{r_{k}^{t_p}} \\
& = \sum_{k=1}^{\infty}  r_k^{t_p/(p-1)}\mu(B^k)^{1-{t_p}/(p-1)}
\simge \sum_{k=1}^{\infty}  r_k^{\be},
\end{align*}
where
\[
  \be
  =\frac{t_p}{p-1}+ \uso\biggl(1-\frac{t_p}{p-1}\biggr)
  = \uso-(\uso-1)\frac{t_p}{p-1}
  = \uso-\uso=0.
\]
Thus the series $\sum_{k=1}^{\infty}  r_k^{\be}$ diverges and hence
$g_u \notin L^{t_p}(\Om)$.

\ref{m-d}
We proceed as in \ref{m-c} up to~\eqref{eq-ak}, 
but this time assuming a
$t_0$-Poincar\'e inequality at $x_0$
for radii up to $2r_0$.
As $p \ge \qhat$, we see that $\lqo \le t_p < p$.
Let
\[
1 < q < \lqo
\quad \text{and} \quad
   \al=\frac{p-t_p}{p-q} = \frac{\uso-p}{(\uso-1)(p-q)}.
\]
Theorem~\ref{thm-comp-u-R1-R2} shows that
\[
a_k \simeq \capp_p(B^k,\Om)^{1/(1-p)} \ge \capp_p(B^k,B^{k-1})^{1/(1-p)},
\]
and 
Corollary~\ref{cor-interpolate-cap}\,\ref{int-b} with $t=t_q$ then implies that
\begin{align}\label{eq-ak-cap-lower}
a_k^{t_p} \ctp(B^{k},B^{k-1})
&\simge \cp(B^{k},B^{k-1})^{t_p/(1-p)} \ctp(B^{k},B^{k-1}) \nonumber \\
&\simge  \capp_p(B^{k},B^{k-1})^{t_p/(1-p)+1-\al}
  \biggl( \frac{\mu(B^k)}{r_k^q} \biggr)^{\alp}.
\end{align}
Note that
\[
\frac{t_p}{1-p} + 1 = \frac{\uso}{1-\uso} + 1 = \frac{1}{1-\uso}<0.
\]
Hence also $t_p/(1-p) + 1-\al<0$
and~\cite[Proposition~5.1]{BBLeh1} gives
\begin{equation}\label{eq-est-capp-simple-plus}
\capp_p(B^{k},B^{k-1})^{t_p/(1-p) + 1-\al} 
\simge \biggl(\frac{\mu(B^k)}{r_k^p}\biggr)^{t_p/(1-p) + 1-\al}
= \biggl(\frac{\mu(B^k)}{r_k^p}\biggr)^{1/(1-\uso)-\al}.
\end{equation}
Since $\al(p-q)(1-\uso)=p-\uso$,
we obtain from \eqref{eq-ak-cap-lower} and \eqref{eq-est-capp-simple-plus} that
\begin{align*}
a_k^{t_p} \ctp(B^{k},B^{k-1}) &\simge 
\biggl(\frac{\mu(B^k)}{r_k^p}\biggr)^{1/(1-\uso)-\al}\biggl( \frac{\mu(B^k)}{r_k^q} \biggr)^{\alp}\\
& = \biggl( \frac{\mu(B^k)}{r_k^{p-\al(p-q)(1-\uso)}} \biggr)^{1/(1-\uso)}
  = \biggl( \frac{\mu(B^k)}{r_k^{\uso}} \biggr)^{1/(1-\uso)} \simge 1, 
\end{align*}
by the choice of $r_k$.
Hence, by applying \eqref{eq-ak}
and summing up the estimates, we conclude that
$g_u \notin L^{t_p}(\Om)$.
\end{proof}

In the following case we get a complete characterization.

\begin{cor}\label{cor-green-full-char}
Let $u$ be a Green function in $\Omega$, with singularity at $x_0$.
Assume that $1<2- 1/\uso \le p  \le \uso \notin \uSo \setm \lSo$
and that one of the following conditions holds\/{\rm:} 
\begin{enumerate}
\item \label{g-a}
$\mu$ supports a $t_0$-Poincar\'e inequality at $x_0$ 
for small radii and some $t_0<t_p$\textup{;}
\item \label{g-b}
$p< \qhat$ and $\mu$ supports a $t_p$-Poincar\'e inequality at 
$x_0$ for small radii\textup{;} 
\item \label{g-c}
$p=\uso$.
\end{enumerate}
Then $g_u \in L^t(\Om)$ if and only if $0 < t < t_p$.
\end{cor}

In particular,
Corollary~\ref{cor-green-full-char} applies if
$1 < 2-1/Q  \le p\le Q$
and $\mu$ is locally Ahlfors $Q$-regular at $x_0$,
as in Theorem~\ref{thm-intro-Q-reg},
and supports 
a $t_p$-Poincar\'e inequality at $x_0$ for small radii,
since in this case $\qhat=\lqo = \uso = Q \in \lSo$.
Under the assumptions in Corollary~\ref{cor-green-full-char},
it follows from
Corollary~\ref{cor-integrability-S-extra} 
that $u \in L^\tau(\Om)$ 
if and only if $\tau < \tau_p$,
i.e.\ we have a full characterization
for the integrability of both $u$ and $g_u$.

\begin{proof}
Since $2- 1/\uso \le p  \le \uso$,
we have $1 \le  t_p\le p$. 
Parts~\ref{g-a} and~\ref{g-b} thus follow from 
Theorems~\ref{thm-integrability-gu} 
and~\ref{thm-nonintegrability-gu}\,\ref{m-c}--\ref{m-d},
while part~\ref{g-c} follows from 
Proposition~\ref{prop-cp-x0}\,\ref{e-a}
and Corollary~\ref{cor-nonint-gu-t=p}\,\ref{a-aa}.
\end{proof}

We have now also completed the proofs
of Theorems~\ref{thm-intro} and~\ref{thm-intro-Q-reg}. 
More precisely, they are deduced as follows.

\begin{proof}[Proof of Theorem~\ref{thm-intro}]
Part~\ref{intro-a} follows 
from 
Proposition~\ref{prop-cp-x0}\,\ref{e-c} and Theorem~\ref{thm-Green-cp-x0},
while parts \ref{intro-b}--\ref{intro-f} follow directly
from
Theorems~\ref{thm-integrability-S}, \ref{thm-nonintegrability-S}\,\ref{hh-a},
\ref{thm-integrability-gu},
Corollary~\ref{cor-nonint-gu-t=p}\,\ref{a-aa}
and Theorem~\ref{thm-nonintegrability-gu}\,\ref{m-a},
respectively.
\end{proof}

\begin{proof}[Proof of Theorem~\ref{thm-intro-Q-reg}]
Formula \eqref{eq-formula-Qreg} follows from 
\eqref{eq-uappr-2},
while \ref{i-a}--\ref{i-c} follow from Theorem~\ref{thm-intro},
together with
Theorem~\ref{thm-nonintegrability-S}\,\ref{hh-b} for \ref{i-a}
and Theorem~\ref{thm-nonintegrability-gu}\,\ref{m-c} for \ref{i-c}.
Note that in this case, $\qhat=\lqo=\uso$.
\end{proof}

\section{Radial weights}
\label{sect-radial-weights}

Radial weights are  
useful for creating examples with various properties 
related to the exponent sets $\lQo$, $\lSo$, $\uSo$
and $\uQo$, see 
Example~\ref{ex-log},
\cite[Section~3]{BBLeh1}, H.~Svensson~\cite{SvenssonH} 
and S.~Svensson~\cite{SvenssonS}.

In this section we take a more detailed look
at the integrability of Green functions for radial weights.
Throughout this section we consider 
$\R^n$, $n \ge 2$, equipped with a radial \p-admissible 
measure $d\mu=w(|x|)\,dx$.

We also let $x_0=0$ 
and $\Om=B_1$, and define for $x\in B_1$,
\begin{equation}\label{eq-radial-green}
   u(x)=\int_{|x|}^1 \bigl(w(\rho) \rho^{n-1}\bigr)^{1/(1-p)} d\rho.
\end{equation}
Then $u$ is \p-harmonic in $B_1 \setm \{0\}$
and \eqref{eq-normalized-Green-intro-deff} holds,
see~\cite[Proposition~10.8 and its proof]{BBLeh1}.
Thus, $u$ is a Green function by \cite[Theorem~8.5]{BBLehGreen}.
The proof of Proposition~10.8 in \cite{BBLeh1} also shows that
\[
   g_u(x)=\bigl(w(|x|)|x|^{n-1}\bigr)^{1/(1-p)}
\]
and thus
\begin{align*} 
   \int_{B_1} g_u^{t}\,d\mu
   &= \omega_{n-1}\int_0^1 \bigl(w(\rho)\rho^{n-1}\bigr)^{-t/(p-1)} w(\rho)\rho^{n-1}\,d\rho \\
   &= \omega_{n-1}\int_0^1 \bigl(w(\rho)\rho^{n-1}\bigr)^{1-t/(p-1)} \,d\rho,
\end{align*}
where $\om_{n-1}$ is the surface area of the $(n-1)$-dimensional sphere
$\Sphere^{n-1}$.

In particular,
\begin{equation} \label{eq-int-w}
   \int_{B_1} g_u^{t_p}\,d\mu
   = \om_{n-1}\int_0^1 \bigl(w(\rho)\rho^{n-1}\bigr)^{1/(1-\uso)} \,d\rho
\end{equation}
which is independent of $p$.
(We do not know if in other
situations the integrability
in the borderline case $t_p$ can
depend on $p$.)
Note that $g_u=|\grad u|$ a.e.\ in $\Om$, by \cite[Proposition~A.13]{BBbook},
where $\grad u$ is the gradient of $u$ in the weighted Sobolev space 
$H^{1,p}\loc(\Om \setm \{x_0\},w)$ as in
Heinonen--Kilpel\"ainen--Martio~\cite[Section~2]{HeKiMa}.
As a consequence we get the following improvement of 
Corollary~\ref{cor-green-full-char} for radial weights, covering also
the case $1<p<2-1/\uso$, i.e.\ when $t_p<1$ 
and thus $g_u\notin L^1(\Om,w)$.

\begin{prop} \label{prop-radial}
Assume that $p\le\uso \notin \uSo\setm \lSo$,
$d\mu=w(|x|)\,dx$ on $\R^n$
and that $\mu$ supports a
$q$-Poincar\'e inequality at $x_0$
for small radii and some $q<\lqo$.
Let $u$ be the radially symmetric Green function 
as in~\eqref{eq-radial-green}. 
Then $g_u \in L^t(\Om,w)$
if and only if
$0 < t < t_p$.
\end{prop}

Note that $u$ of course depends on $p$. 
Moreover, the assumption that $w$ is \p-admissible 
may implicitly further limit the range of $p$
in Proposition~\ref{prop-radial}.

\begin{proof}
By continuity we
find $1<p_0<\uso$ so that $t_{p_0}=q$.
Since $q < \lqo$, it follows that $p_0 < \qhat$.
Thus Theorem~\ref{thm-nonintegrability-gu}\,\ref{m-c} shows
that $g_u \notin L^{t_{p_0}}(\Om,w)$ for $p=p_0$.
By \eqref{eq-int-w}, we get that
$g_u \notin L^{t_{p}}(\Om,w)$ is true for all $1< p \le \uso$.
The integrability for $t<t_p$ follows from 
Theorem~\ref{thm-integrability-gu}.
\end{proof}

\begin{example} \label{ex-log-2}
Consider the weight $w$ as in Example~\ref{ex-log},
i.e.\
\[
      w(\rho)=\begin{cases}
          \rho^{s-n}  |{\log \rho}|^\be, & \text{if } 0<\rho\le1/e, \\
          \rho^{s-n}, & \text{otherwise},
        \end{cases}
\]
where  $s>1$ and  $\be >0$.
As observed in Example~\ref{ex-log}, $w$ 
is $1$-admissible
on $\R^n$,
\[
\lso=\lqo=\uso=\uqo=s, \quad
\lSo=\lQo=(0,s)
\quad \text{and} \quad
\uSo=\uQo=[s,\infty).
\]
Then 
the case $t=t_p$ 
is not covered by Theorem~\ref{thm-nonintegrability-gu},
but all other exponents are covered
by either Theorem~\ref{thm-integrability-gu} 
or~\ref{thm-nonintegrability-gu}\,\ref{m-a}, when $p<s$.
For $t=t_p$, \eqref{eq-int-w} simplifies to
\[
   \int_{B_1} g_u^{t_p}\,d\mu
   \simeq \int_0^{1/e} \frac{d\rho}{\rho |{\log \rho}|^{\be/(s-1)}}, 
\]
which converges if and only if $\be > s-1$.
So if $0 < \be \le s-1$, then $g_u \notin L^{t_p}(\Om,w)$,
while $u \in L^{\tau_p}(\Om,w)$ 
if and only if $s/(1+\be) <p< s$, by Example~\ref{ex-log}.
Since $t_p=p\tau_p/(\tau_p+1)$, this shows that the strict inequalities in
Theorem~\ref{thm-int-Gu-superh-new} cannot be replaced by nonstrict ones,
when $\tau_p<\infty$.

On the other hand, for $p=s$, Theorem~\ref{thm-Green-cp-x0} 
shows that $u \in L^{\tau_p}(\Om,w)$ (i.e.\ $u$ is bounded)
if and only if $g_u \in L^p(\Om,w)$.
According to Example~\ref{ex-log}, this is equivalent to $\be > s-1$.
\end{example}

\section{General \texorpdfstring{\p}{p}-harmonic functions
with poles}
\label{sect-poles}

\emph{Recall the general assumptions from
the beginning of Section~\ref{sect-harm}.}

\medskip

In~\cite[Theorem~10.1]{BBLehGreen} it was shown that
any \p-harmonic function in $\Om \setm \{x_0\}$ with a pole at $x_0$
has similar growth properties near the pole as Green functions
with singularity at $x_0$. Hence the results in the previous sections
can be extended to such functions, in a local sense.

\begin{thm} \label{thm-poles-int}
Assume that $u$ is a \p-harmonic function in $\Om \setm \{x_0\}$
such that $u(x_0):=\lim_{x \to x_0} u(x)=\infty$.
Then the following are true\/\textup{:}
\begin{enumerate}
\item \label{r-a}
  $\Cp(\{x_0\})=0$\textup{;}
\item \label{r-b}
  $p \le \uso$\textup{;}
\item \label{r-c}
  $u \in L^\tau\loc(\Om)$ for all $0 < \tau < \tau_p$\textup{;}
\item \label{r-d}
  $u \notin L^\tau\loc(\Om)$ if $\tau > \tau_p$\textup{;}
\item \label{r-e}
if $\uso \notin \uSo\setm \lSo$, 
then $u \notin L^{\tau_p}\loc(\Om)$\textup{;}
\item \label{r-f}
$u \in L^{\tau_p}\loc(\Om)$ if and only if 
$p < \uso$ and \eqref{eq-prop-u-Q} holds,
or $p = \uso$ and \eqref{eq-prop-u-1} holds\textup{;}
\item \label{r-g}
$g_u\in L^t\loc(\Om)$ for all $0<t<t_p$\textup{;}
\item \label{r-h}
if $p = \uso$, then 
$g_u \in L^t\loc(\Om)$ if and only if $0<t<p$\textup{;}
\item \label{r-i}
if $p <\uso$, then $g_u\notin L^t\loc(\Om)$ 
whenever $\mu$ supports a $t$-Poincar\'e inequality at $x_0$ 
for small radii, $t>t_p$ and $t \ge 1$\textup{;}
\item \label{r-j}
if $p< \qhat$, $t_p \ge 1$ and $\uso \notin \uSo\setm \lSo$, 
then 
$g_u\notin L^{t_p}\loc(\Om)$ whenever 
$\mu$ supports a $t_p$-Poincar\'e inequality at $x_0$
for small radii\textup{;}
\item \label{r-k}
if $\qhat \le p< \uso\notin \uSo\setm \lSo$ and $\lqo>1$, then
$g_u\notin L^{t_p}\loc(\Om)$ whenever 
$\mu$ supports a $t_0$-Poincar\'e inequality at $x_0$ 
for small radii and some $1 \le t_0<t_p$.
\end{enumerate}
\end{thm}

\begin{proof}
By Theorems~1.3\,(b) and~10.1\,(b)
in \cite{BBLehGreen}, there 
are $a>0$, $b\in\R$ 
and a bounded domain $U \subset \Om$ such that $x_0 \in U$ and 
$v:=au+b$ is a Green function in $U$ with singularity at $x_0$.

As $u$ is \p-harmonic in $\Om \setm \{x_0\}$
it is locally bounded therein
and also $g_u \in L^p\loc(\Om \setm \{x_0\})$.
It is therefore enough to consider the integrability
and nonintegrability conditions on $U$ 
(since we only consider $L^t\loc$-integrability of $g_u$ for $t \le p$).
Note that $g_v=ag_u$.

\ref{r-a} This follows from Theorem~10.1\,(a) in \cite{BBLehGreen}.

\ref{r-b} This follows from \ref{r-a} and 
Proposition~\ref{prop-cp-x0}\,\ref{e-c}.

\ref{r-c} This follows from Theorem~\ref{thm-integrability-S}.

\ref{r-d} and \ref{r-e}
These statements follow from Theorem~\ref{thm-nonintegrability-S}.

\ref{r-f} This follows from Theorem~\ref{thm-u-Q}.

\ref{r-g} This follows from Theorem~\ref{thm-integrability-gu}.

\ref{r-h} This follows from  Corollary~\ref{cor-nonint-gu-t=p}.

\ref{r-i}--\ref{r-k} 
These statements follow from Theorem~\ref{thm-nonintegrability-gu}.
\end{proof}

\begin{thm}   \label{thm-est-pole}
Assume that $u\ge 0$ is a \p-harmonic function in $B_{50\la R} \setm \{x_0\}$
such that $u(x_0):=\lim_{x \to x_0} u(x)=\infty$.
Then for all $0<r<R/50\la$ and $x\in S_r$,
\begin{equation}   \label{eq-pot-est-pole}
u(x) \simeq A \biggl( \inf_{B_{R}} u 
    + \int_r^{R} \biggl( \frac{\rho}{\mu(B_\rho)} \biggr)^{1/(p-1)} \,d\rho \biggr),
\end{equation}
where the implicit comparison constants
are independent of $u$, while $A$ depends on
$u$ only as follows,
\begin{equation}   \label{eq-def-al}
A = \biggl( \int_{a<u<a+1} g_u^p\,d\mu \biggr)^{1/(p-1)}
\quad \text{for any } a\ge \max_{\bdy B_{R}} u.
\end{equation}
\end{thm}

Note that 
$A=\cp(G_{a+1},G_a)^{1/(p-1)}$, 
where $G_a=\{y\in B_R : u(y)> a\}$.
The proof below shows that the integral
in~\eqref{eq-pot-est-pole} can be replaced by $\cp(B_r,\Om_0)^{1/(1-p)}$,
where
\[
\Om_0=\{x\in B_R:u(x)>\max_{\bdy B_{R}} u\}.
\]
This integral
can thus be estimated by $\cp(B_r,B_R)^{1/(1-p)}$ and $\cp(B_r,B_{R_1})^{1/(1-p)}$
from above and below, respectively, whenever $B_{R_1} \subset \Om_0$.
In particular, Theorem~\ref{thm-est-pole} generalizes the estimates obtained
for \p-harmonic Green functions in
Danielli--Garofalo--Marola~\cite[Lemma~5.1]{DaGaMa}.

\begin{proof}[Proof of Theorem~\ref{thm-est-pole}]
Proposition~4.4 in~\cite{BBLehGreen} implies that 
\begin{equation} \label{eq-k0-on-BR}
k_0 := \max_{\bdy B_{R}} u\simeq \min_{\bdy B_R} u = \inf_{B_R} u,
\end{equation}
with the implicit comparison constants depending only on  $p$,
the doubling constant of $\mu$ and the constants in the Poincar\'e inequality.
By \cite[Theorem~1.6]{BBLehGreen}, $u-k_0$ is a singular function in
$\Om_0=\{x\in B_R:u(x)>k_0\}$,
and thus $\cp(\{x_0\},\Om)=\Cp(\{x_0\})=0$, by
Theorem~\ref{thm-Green-cp-x0}.
Hence, \cite[Theorem~9.3]{BBLehGreen} implies that
$\ut=(u-k_0)/A$ is a Green function in $\Om_0$,
where $A>0$ is given 
by~\eqref{eq-def-al}.
Let $B_{R_1}\subset\Om_0$.
Applying \eqref{eq-uappr-0} to $\ut$ and $\Om_0$, instead of $u$ and $\Om$, 
shows that for all $0<r<R_1/50\la$ and $x\in S_r$,
\begin{equation}   \label{eq-est-v=u-k0}
k_0 + C_1 A \cp(B_r,\Om_0)^{1/(1-p)} \le u(x) 
\le k_0 + C_2 A \cp(B_r,\Om_0)^{1/(1-p)},
\end{equation}
where $C_1, C_2>0$ depend only on $p$, 
the doubling constant of $\mu$ and the constants in the Poincar\'e inequality,
but not on $R_1$, $u$, $x$ or $\Om_0$.

By Theorem~6.3 in Bj\"orn~\cite{ABremove} (or \cite[Lemma~4.3]{BBLehGreen}),
$u$ is superharmonic in $B_{50\la R}$.
As $u$ is nonconstant, \cite[Corollary~9.14]{BBbook} shows
that $X \ne B_{50\la R}$,
and thus  $50 \la R \le \diam X$.
Since $B_{R_1}\subset\Om_0\subset B_R$, we therefore conclude from 
Theorem~\ref{thm-metaestimate-int-only} that
\begin{align}  \label{eq-cap-Om0-int}
\cp(B_r,\Om_0)^{1/(1-p)} &\le \cp(B_r,B_R)^{1/(1-p)} 
\simeq \int_r^{R} \psi(\rho) \,d\rho, \nonumber \\
\cp(B_r,\Om_0)^{1/(1-p)} &\ge \cp(B_r,B_{R_1})^{1/(1-p)} 
\simeq \int_r^{R_1} \psi(\rho) \,d\rho,
\end{align}
where
\[
\psi(\rho) = \biggl( \frac{\rho}{\mu(B_\rho)} \biggr)^{1/(p-1)}.
\]
Moreover, as $\cp(\{x_0\},\Om_0)=0$, 
Theorem~\ref{thm-metaestimate-int-only} also implies that
$\int_0^{R_1} \psi(\rho) \, d\rho=\infty$.
Thus, there exists $r_1< R_1/50\la$ such that
for all $0<r\le r_1$, the integral in~\eqref{eq-cap-Om0-int}
satisfies
\[
\int_{r}^{R_1} \psi(\rho) \,d\rho
= \int_{r}^{R} \psi(\rho) \,d\rho - \int_{R_1}^{R} \psi(\rho) \,d\rho
\ge \frac12 \int_{r}^{R} \psi(\rho) \,d\rho.
\]
Inserting this into \eqref{eq-cap-Om0-int} and \eqref{eq-est-v=u-k0},
together with  \eqref{eq-k0-on-BR},
proves \eqref{eq-pot-est-pole} for $r\le r_1$.

In order to obtain \eqref{eq-pot-est-pole} for $r < R/50\la$, 
let $v$ be a Green function in $B_R$ with singularity at $x_0$,
and let $v=0$ outside $B_R$.
In particular, $v \in \Nploc(X \setm \{x_0\})$.
Then by \eqref{eq-uappr-2}, applied to $v$ and $\Om=B_R$, 
we have 
\begin{equation}  \label{eq-est-green-v-in-BR}
v(x) \simeq \int_{r}^{R} \psi(\rho) \,d\rho,
\quad \text{where }  r=d(x,x_0)<R/50\la,
\end{equation}
with the implicit constants depending only on $p$, 
the doubling constant of $\mu$ and the constants in the Poincar\'e inequality.
In particular, together with the already proved \eqref{eq-pot-est-pole} 
for $r\le r_1$, 
we have 
\[
A v \simeq u-k_0 \le u-m \quad \text{on $\bdy B_{r_1}, \quad$
where  $m:=\min_{\bdy B_R}u \le k_0$}.
\]
Since $u-k_0 \le 0=A v \le u-m$ on $\bdy B_R$, the comparison principle 
for \p-harmonic functions
with Sobolev boundary values (\cite[Lemma~8.32]{BBbook})
implies that
$u-k_0 \simle A v \simle u-m$ in $B_R \setm B_{r_1}$.
Thus, \eqref{eq-est-green-v-in-BR} and \eqref{eq-k0-on-BR}
show that \eqref{eq-pot-est-pole} holds for all $r < R/50\la$.
\end{proof}

\begin{remark} \label{rmk-Wolff}
The estimate in Theorem~\ref{thm-est-pole} is related to the so-called Wolff 
potential, which is a nonlinear analogue of the Riesz potential.
More precisely, Theorem~1.6 in Kilpel\"ainen--Mal\'y~\cite{KilMa-Acta}
shows that if $u\ge0$ is a superharmonic
function in  $B(x,3R)\subset\R^n$ (unweighted)
and $\nu=-\Delta_p u$ is its Riesz measure, then
\begin{equation} \label{eq-Wolff}
W^\nu_{1,p} (x,R) \simle
u(x) \simle
\inf_{B(x,R)}u + W^\nu_{1,p} (x,2R),
\end{equation}
where 
\[
W^\nu_{1,p} (x,R) 
= \int_0^R \biggl( \frac{\nu(B(x,\rho))}{\rho^{n-p}} \biggr)^{1/(p-1)} \,\frac{d\rho}{\rho}
\]
is the Wolff potential, see also 
Maz\cprime ya--Havin~\cite{MazHa70}, \cite[Theorem~6.1]{MazHa72}.

In the case of Green functions, $\nu=\de_{x_0}$ is the Dirac measure at $x_0$
and hence
\[
W^{\de_{x_0}}_{1,p} (x,R) 
= \int_{r}^R \rho^{(p-n)/(p-1)-1} \,d\rho \quad \text{when }
r=|x-x_0|<R,
\]                
which is comparable to the integral in \eqref{eq-pot-est-pole}, since 
the Lebesgue measure of $B(x,\rho)$ is a multiple of $\rho^n$.
Thus, our Theorem~\ref{thm-est-pole} generalizes the Wolff potential estimate
for the Dirac measure to \p-harmonic functions with poles in metric spaces, 
even in the case when there is no \p-harmonic equation.
\end{remark}

\begin{proof}[Proof of Theorem~\ref{thm-intro-gen-pharm-new}]
This follows directly from Theorems~\ref{thm-poles-int},
\ref{thm-est-pole} and \ref{thm-comp-u-R1-R2}.
\end{proof}

\section{Elliptic equations in divergence form}
\label{sect-elliptic-eq}

Estimates similar to \eqref{eq-Wolff}
also hold for elliptic differential
equations in divergence form
of the type
\begin{equation}   \label{eq-div-form-Du}
\Div A(x,\grad u) = 0,
\end{equation}
including weighted and vectorial ones, see
Kilpel\"ainen--Mal\'y~\cite{KilMa-Pisa}, \cite{KilMa-Acta},
Mikkonen~\cite{Mikkonen}, Hara~\cite{Hara-Rn},
the monograph Heinonen--Kilpel\"ainen--Martio~\cite[Theorem~21.21]{HeKiMa} and the 
expository papers Kuusi--Mingione~\cite{KuuMin14}, \cite{KuuMin18}.
On metric spaces, such estimates
have been obtained for Cheeger \p-harmonic functions
in Bj\"orn--MacManus--Shan\-mu\-ga\-lin\-gam~\cite{BMS}
and Hara~\cite{Hara}.

In this section 
we
explain how to extend our (non)integrability
results from energy minimizers to  Green  functions 
for \eqref{eq-div-form-Du},
i.e.\ functions $u$ satisfying
\[
\Div A(x,\grad u) = -\delta_{x_0}
\]
in $\Om$ 
with zero boundary values on $\bdy\Om$ (in Sobolev sense).
Here, $\grad u$ stands for one of the following gradients:
\begin{itemize}
\item Usual (distributional) gradient in unweighted $\R^n$
or the gradient defined for the weighted Sobolev space
$H^{1,p}\loc(\Om,w)$ as in
  Heinonen--Kilpel\"ainen--Martio~\cite[Section~1.9]{HeKiMa}.
\item Natural gradient $\grad u$ on Riemannian manifolds with nonnegative
  Ricci curvature as in Holopainen~\cite{Ho}, \cite{Holo-pos} and
\cite{HoDuke}. 
\end{itemize} 

The vector-valued function $A(x,\xi)$ is for a.e.\ $x$ and all $\xi$
assumed to satisfy the usual ellipticity conditions as in 
\cite[(3.3)--(3.7)]{HeKiMa}, Fabes--Jerison--Kenig~\cite{FaJeKe}
or Holopainen~\cite[(2.9)--(2.12)]{Holo-pos}.
Subelliptic equations associated with left-invariant  
vector fields 
\[
Xu=(X_1u,\ldots, X_ku)
\]
in Heisenberg or Carnot groups, 
as in  Haj\l asz--Koskela~\cite[Sections~11.3 and 11.4]{HaKo}
or Capogna--Danielli--Garofalo~\cite{CaDaGa2},
could also be considered with obvious interpretations,
whenever the main 
ingredients, specified below, are satisfied.

For the integrability of the Green functions
for $\Div A(x,\grad u) = 0$, 
it is sufficient to
invoke the capacitary estimate 
\begin{equation}  \label{eq-est-u-cap}
u(x) \simeq \cpmu(B_r,\Om)^{1/(1-p)} \quad \text{for sufficiently small } 
  r:=d(x,x_0)>0,
\end{equation}
which has been proved
in weighted $\R^n$ in \cite[Lemma~3.1]{FaJeKe} (for $p=2$)
and in \cite[Theorem~7.41]{HeKiMa} (for balls and $1<p<\infty$),
while on manifolds it follows from the proof of Theorem~3.19 in \cite{Ho}.

In view of \eqref{eq-uappr-0}, estimate~\eqref{eq-est-u-cap} 
implies that $u$ is in a
neighbourhood of $x_0$ comparable to the Green function considered in
this paper and thus has the same (non)integrability properties.
 
We can also obtain (non)integrability results
for the gradient $\grad u$.
In addition to \eqref{eq-est-u-cap} and our results 
in the previous sections, the only required tools are the minimum principle 
and the Caccioppoli inequality 
\begin{equation}   \label{eq-Caccioppoli}
\int_B |\grad v|^p v^{-(1+\eps)}\,d\mu \le C \int_{2B} v^{p-(1+\eps)}\,d\mu
\end{equation}
for positive supersolutions $v$ of \eqref{eq-div-form-Du} in $2B$
and for $\eps>0$,
with $C$ independent of $v$; this will be applied to truncations of $u$.
Such inequalities are well known in the cases considered
above, see e.g.\ \cite[Lemma~3.57]{HeKiMa} and
\cite[Lemma~3.1]{Holo-pos}. 
(In \cite[Lemma~3.1]{Holo-pos}, this is proved for solutions 
of \eqref{eq-div-form-Du}, but the proof goes
through verbatim also for supersolutions when $q=p-(1+\eps)$.)

For simplicity, we formulate and prove the next result only 
in weighted $\R^n$ with a \p-admissible weight as in \cite{HeKiMa}.
The case of Riemannian manifolds with nonnegative Ricci curvature 
(so that the doubling property and the \p-Poincar\'e inequality hold,
see Haj\l asz--Koskela~\cite[Chapter~10.1]{HaKo})  
is similar with obvious modifications.
For the definition and existence of singular and Green functions 
in these settings we refer to 
\cite[Section~7.38]{HeKiMa} and \cite[Definition~3.9 and Theorem~3.19]{Ho}.
Since the capacitary estimate \eqref{eq-est-u-cap} is in
\cite[Theorem~7.41]{HeKiMa} proved only when $\Om$ is a ball, we
restrict ourselves to this case.
(In \cite[Theorem~7.41]{HeKiMa} the assumption that $h$ is
nonnegative should be added.)

\begin{thm}  \label{thm-sum-A-superh}
Let $w$ be a \p-admissible weight on $\R^n$, $d\mu=w\,dx$
and assume that $A\colon \R^n\times\R^n\to \R^n$ satisfies the ellipticity 
conditions {\rm(3.3)--(3.7)} in~{\rm\cite{HeKiMa}}.
Assume that the capacity $\Cpmu(\{x_0\})=0$ 
and 
let $u\in H^{1,p}\loc(B_R \setm \{x_0\},w)$ 
be a continuous weak solution of \eqref{eq-div-form-Du} in
$B_R\setm\{x_0\}$ 
such that 
\begin{equation}   \label{eq-lim=infty}
u(x_0)=\lim_{x \to x_0} u(x)=\infty.
\end{equation}
Assume that $u$ has zero boundary
values on $\bdy B_R$, either in the Sobolev sense
\ref{dd-Np0} or as the limit
\begin{equation} \label{eq-sing-soln}
   \lim_{\Om \ni y \to x} u(x)=0
   \quad \text{for all } x \in \bdy B_R.
\end{equation}
Then the conclusions in Section~\ref{sect-int-Green} hold for $u$
in $\Om:=B_R$. 
Also Theorems~\ref{thm-integrability-gu}
and~\ref{thm-nonintegrability-gu}
hold for $g_u=|\grad u|$.
\end{thm}

\begin{remark} \label{rem-mult-Green}
In fact, the assumptions \ref{dd-Np0} and \eqref{eq-sing-soln} in
Theorem~\ref{thm-sum-A-superh} are equivalent.
Indeed, \cite[Theorem~6.31]{HeKiMa} 
(applied to $B_R \setm \itoverline{B}_{R/2}$) shows that \ref{dd-Np0}
implies  \eqref{eq-sing-soln}.
The converse implication follows from the uniqueness Theorem~3.12 
in Bj\"orn--Bj\"orn--Mwasa~\cite{BBMwasa}, together with 
\cite[Corollary~9.29]{HeKiMa} applied to the boundary data
$f:=u\eta\in H^{1,p}(B_R \setm \itoverline{B}_{R/2},w)$ with a
suitable cut-off function $\eta\in C^\infty_0(B_R)$.

Moreover, Harnack's inequality on spheres $S_r$ implies that for
nonnegative $u$ the
limit in \eqref{eq-lim=infty} can be replaced by $\limsup$.
The removability Theorem~7.35 in \cite{HeKiMa} shows that $u$ is
$A$-superharmonic in the ball $B_R$.
Hence, by Theorem~2.2 in Mikkonen~\cite{Mikkonen}
(or \cite[Theorem~21.2]{HeKiMa}), there exists a Riesz measure $\mu$
such that $\Div A(x,\grad u) = -\mu$.
As $u$ is $A$-harmonic in $B_R \setm \{x_0\}$ the Riesz measure
must be concentrated to $\{x_0\}$.
Therefore
a multiple of $u$ is a Green function for \eqref{eq-div-form-Du},
as defined in the beginning of this section.
\end{remark}

\begin{proof}[Proof of Theorem~\ref{thm-sum-A-superh}]
Proposition~A.17 in \cite{BBbook} implies that $X=(\R^n,\mu)$
satisfies the general assumptions from the beginning of
Section~\ref{sect-harm}.
Theorem~7.41 in \cite{HeKiMa} shows that \eqref{eq-est-u-cap},
and hence also \eqref{eq-uappr-0}, holds.
Thus $u$ has the same (non)integrability properties near $x_0$ as the
\p-harmonic Green functions studied in this paper for bounded domains
in $X$. 
Moreover, $u$ is locally bounded in
$B_R\setm\{x_0\}$. 
Together with \eqref{eq-sing-soln}, this
proves the results from Section~\ref{sect-int-Green} for~$u$.

To obtain the integrability of $\grad u$, note that
the proof of Theorem~\ref{thm-int-Gu-superh-new} shows that statement
\ref{kmm-a} therein holds for $u$, since its shifted
truncations  $u_k$ in~\eqref{eq-def-trunc}
are supersolutions and thus satisfy the Caccioppoli
inequality \eqref{eq-Caccioppoli} whenever $2B\Subset\Om$.
Combined with the integrability of $u$ itself, we conclude that 
$\grad u\in L^t\loc(\Om)$ when $0<t<t_p$, as in 
Theorem~\ref{thm-integrability-gu}.
As $u \in H^{1,p}(B_R \setm \itoverline{B}_{R/2},w)$ by \ref{dd-Np0},
we conclude  that $\grad u\in L^t(\Om)$ when $0<t<t_p$.

For the nonintegrability of $\grad u$, we follow the proof of
Theorem~\ref{thm-nonintegrability-gu}. 
The functions $\min\{1,u/m_r\}$ and $v_k$, considered therein, are
admissible for $\capp_{t,\mu}(B_r,\Om)$ and $\capp_{t_p,\mu}(B^k,B^{k-1})$ also
in this case. 
The minimum principle holds for $u$ by \cite[Theorem~7.12]{HeKiMa}.
Since the rest of the proof depends only on the capacity estimates
\eqref{eq-lower-for-m_r} and \eqref{eq-ak-cap-lower}, provided in this
case by~\eqref{eq-est-u-cap},
the conclusions of Theorem~\ref{thm-nonintegrability-gu} follow.
\end{proof}

\end{document}